\documentclass{article} 

\usepackage{hyperref}

\hypersetup{
colorlinks=false,
linkcolor=red, 
citecolor=green
}

\usepackage{apacite}

\usepackage{amsthm}

\usepackage{graphicx}

\usepackage{appendix}

\theoremstyle{theorem}
\newtheorem{thm}{Theorem}
\newtheorem{lemma}[thm]{Lemma}
\newtheorem{prop}[thm]{Proposition}

\newtheorem{algm}{Algorithm}

\theoremstyle{remark}
\newtheorem{assumption}{Assumption}
\newtheorem{rmk}{Remark}

\theoremstyle{definition}
\newtheorem{defn}{Definition}
\newtheorem{example}{Example}

\usepackage{amsmath}
\usepackage{amssymb}
\usepackage{bbm}

\usepackage[T1]{fontenc}

\usepackage{color}
\usepackage{soul}

\newcommand\omicron{o}

\newcommand{\I}[1]{\mathbbm 1_{\{#1\}}}
\newcommand{\E}{\mathbb{E}}
\renewcommand{\P}{\mathbb{P}}

\title{Importance sampling of heavy-tailed iterated random functions}
\author{
        By Bohan Chen, Chang-Han Rhee \& Bert Zwart \\\\
        Centrum  Wiskunde \& Informatica\\
        Science Park 123, 1098 XG Amsterdam, Netherlands
}
\date{\today}

\begin{document}
\maketitle

\begin{abstract}
We consider a stochastic recurrence equation of the form $Z_{n+1} = A_{n+1} Z_n+B_{n+1}$, where $\mathbb{E}[\log A_1]<0$, $\mathbb{E}[\log^+ B_1]<\infty$ and $\{(A_n,B_n)\}_{n\in\mathbb{N}}$ is an i.i.d.\ sequence of positive random vectors. The stationary distribution of this Markov chain can be represented as the distribution of the random variable $Z \triangleq  \sum_{n=0}^\infty B_{n+1}\prod_{k=1}^nA_k$. Such random variables can be found in the analysis of probabilistic algorithms or financial mathematics, where $Z$ would be called a stochastic perpetuity. If one interprets $-\log A_n$ as the interest rate at time $n$, then $Z$ is the present value of a bond that generates $B_n$ unit of money at each time point $n$. We are interested in estimating the probability of the rare event $\{Z>x\}$, when $x$ is large; we provide a consistent simulation estimator using state-dependent importance sampling for the case, where $\log A_1$ is heavy-tailed and the so-called Cram\'{e}r condition is not satisfied. Our algorithm leads to an estimator for $P(Z>x)$. We show that under natural conditions, our estimator is strongly efficient. Furthermore, we extend our method to the case, where $\{Z_n\}_{n\in\mathbb{N}}$ is defined via the recursive formula $Z_{n+1}=\Psi_{n+1}(Z_n)$ and $\{\Psi_n\}_{n\in\mathbb{N}}$ is a sequence of i.i.d.\ random Lipschitz functions.
\end{abstract}

\section{Introduction}\label{sec1}
We consider an $\mathbb{R}$-valued Markov chain $\{Z_n\}_{n\in\mathbb{N}}$ defined by
\begin{equation}\label{sec1eq1}
Z_{n+1} = \Psi_{n+1}(Z_n),
\end{equation}
where $\{\Psi_n\}_{n\in\mathbb{N}}$ is a sequence of i.i.d.\ positive random Lipschitz functions and $Z_0\in\mathbb{R}$ is arbitrary but independent of the sequence $\{\Psi_n\}_{n\in\mathbb{N}}$. For $k<n$, define the backward iteration as
$$\Psi_{k:n}(z) \triangleq  \Psi_{k}\circ\Psi_{k+1}\circ\cdots\circ\Psi_{n}(z).$$
Define $Z^{(n)}(z_0)\triangleq \Psi_{1:n}(z_0)$. Under some mild conditions (cf.\ Assumption \ref{sec3ass2} below), the sequence $\{Z^{(n)}(z_0)\}_{n\in\mathbb{N}}$ converges a.s.\ to some random variable $Z$. Moreover, this limit does not depend on the choice of the initial condition $z_0$ (cf.\ \citeNP[Theorem 3.1]{dyszewski2016}) and has the same distribution as the stationary solution to \eqref{sec1eq1}. For simplicity we set
\begin{equation}\label{sec1eq5}
Z^{(n)} \triangleq  Z^{(n)}(0).
\end{equation}
We assume that $\Psi_n$ is such that $Z^{(n)}$ is increasing in $n$. Define $T(x)=\inf\{n\geq0:\,Z^{(n)}>x\}$. This paper develops efficient simulation methods for estimating the tail probability of $Z$, i.e.\ we are interested in computing
$$\mathbb{P}(Z>x)=\mathbb{P}(T(x)<\infty),$$
when $x$ is large.

The main example we have in mind, is the so called stochastic perpetuity (also known as infinite horizon discounted reward). More precisely, we consider the random difference equation, where $\Psi_n$ is an affine transformation, that is $\Psi_{n}(z) = A_n z+B_n$. The formula \eqref{sec1eq1} can be written as
\begin{equation}\label{sec1eq2}
Z_{n+1}=A_{n+1}Z_n+B_{n+1},
\end{equation}
for $n\in\mathbb{N}$. It is well known that if $\mathbb{E}[\log A_1]<0$ and $\mathbb{E}[\log^+ B_1]<\infty$, then the Markov chain given by \eqref{sec1eq2} has a unique stationary distribution, which has the same distribution as the random variable
$$Z \triangleq  \sum_{n=0}^\infty B_{n+1} e^{S_n},$$
where $S_n\triangleq S_{n-1}+X_n$, $S_0\triangleq 0$ and $X_n\triangleq \log A_n$. Such random variables can be found in the analysis of probabilistic algorithms or financial mathematics, where $Z$ would be called a stochastic perpetuity. If one interprets $-\log A_n$ as the interest rate at time $n$, then $Z$ is the present value of a bond that generates $B_n$ unit of money at each time point $n$. Perpetuities also occur in the context of ruin problems with investments, in the study of financial time series such as ARCH-type processes (cf.\ e.g.\ \citeNP{embrechtsklueppelbergmikosch1997}), in tail asymptotics for exponential functionals of L\'{e}vy processes (see e.g.\ \citeNP{maulikzwart2006}), etc. A book devoted to \eqref{sec1eq2} is \citeA{buraczewskidamekmikosch2016}. Although some particular cases exist that allow for an explicit analysis (see e.g.\ \citeNP{vervaat1979}), it is hard to come up with exact results for the distribution of $Z$ in general. Thus Monte Carlo simulation arises as a natural approach to deal with the analysis of stochastic perpetuities, including the large deviations regime where $x$ in $\mathbb{P}(T(x)<\infty)$ is large, which is the focus of this paper.

In this paper we develop a state-dependent importance sampling algorithm that can be proved to be strongly efficient. By state-dependent, we mean that the importance sampling distribution for generating $Z^{(n+1)}$ is dependent on the current state $Z^{(n)}$. We say that an estimator $L(x)$ for $\mathbb{P}(T(x)<\infty)$ is strongly efficient (for a discussion of efficiency in rare-event simulation, see e.g.\ \citeNP{asmussenglynn2007}) if
\begin{equation}\label{sec1eq3}
\sup_{x>1}\frac{\mathbb{E}L^2(x)}{\mathbb{P}(T(x)<\infty)^2}<\infty.
\end{equation}

To explain the idea behind our algorithm, consider a stochastic perpetuity, where $B_n=1$. One difficulty that arises in our setting---where $\log A_1$ is heavy-tailed---is that the  Cram\'{e}r condition is not satisfied (a study of the Cram\'{e}r case can be found in \citeNP{blanchetlamzwart2012}), and hence, standard techniques such as exponential change of measure cannot be used. The algorithm we provide in the present paper is based on the fact that the stochastic perpetuity is closely related to the maximum of a random walk. More precisely, for $\gamma\in(0,-\mathbb{E}X_1)$ we observe that
\begin{equation}\label{sec1eq4}
Z = \sum_{n=0}^\infty \exp(S_n)=\sum_{n=0}^\infty \exp\{S_n+n\gamma\}\exp(-n\gamma)\leq \exp\left\{\max\limits_{n\geq0}\left(S_n+n\gamma\right)\right\}\frac{1}{1-e^{-\gamma}}.
\end{equation}
For a general sequence $\{\Psi_n\}_{n\in\mathbb{N}}$, we construct a slightly more involved upper bounding random walk and use it to construct a coupling. This allows us to leverage an importance sampling algorithm designed for random walks, in \citeA{blanchetglynn2008}. We can extend this idea to a general Markov chain given by \eqref{sec1eq1}. The tail asymptotics in this case have been derived by \citeA{dyszewski2016}. Our extension of \eqref{sec1eq4} leads to a shorter proof of the asymptotic upper bound given in that paper.

Note that $Z$ is defined over an infinite horizon, and hence, requires an infinite amount of computational effort for generating each sample; a natural approach to address such an issue is to work with approximations by finite-time truncation. In this paper we study the bias introduced by such approximations and show that our estimator has a vanishing relative bias as $x\to\infty$. We also study the asymptotic behavior of the bias with respect to the truncation time. We show that the relative bias converges to $0$ at a polynomial rate, which depends on the moment condition of $\log A_1$. It should be mentioned that such a convergence rate is due to the heavy-tailed nature of $\log A_1$; in case Cram\'{e}r condition holds for $\log A_1$, geometric convergence rates typically ensue, i.e, the relative bias converges exponentially to $0$ (cf.\ e.g.\ \citeNP[Theorem 2.8]{basrakdavismikosch2002}). By identifying such a convergence rate and proving a uniform bound on the relative moment of the associated estimator (slightly stronger statement than the strong efficiency proved in \citeNP{blanchetglynn2008}), we show that one can apply the bias elimination technique studied in \citeA{rheeglynn2015} to construct strongly efficient and unbiased estimators. 

The rest of the paper is organized as follows. In Section 2 we briefly review an efficient simulation algorithm for the maximum of heavy-tailed random walks proposed by \citeA{blanchetglynn2008}. In Section 3 we introduce the random objects that can be handled by our algorithm. In Section 4 and 5 we use the result from Section 3 to prove an asymptotic upper bound of $\mathbb{P}(Z>x)$. Moreover, we construct an efficient algorithm for estimating $\mathbb{P}(Z>x)$ and analyze its relative bias. The main result of this paper is given as Theorem \ref{sec4thm1}. In Section 6 we analyze the asymptotic behavior of the bias with respect to the truncation time, based on a simple example; we apply the method studied in \citeA{rheeglynn2015} to obtain an unbiased estimator. In Section 7 we present our computational results.

\section{Notations and Preliminary Results}\label{sec2}
In this section we will first introduce several notations, then we will recall some well known preliminary results.

Let $x^+=\max(x,0)$ denote the positive part of $x$ and let $\log^+(x)=\max(\log x,0)=\log\left(\max\left(x,1\right)\right)$. We first recall the following lemma, which will be very useful in validating our new estimator.
\begin{lemma}[\citeNP{glynn2012}]\label{sec2lem1}
Let $\{Y_n\}_{n\in\mathbb{N}}$ be a sequence of random variables on the probability space $(\Omega,\mathcal{F},\mathbb{P})$. Let $\{M_n\}_{n\in\mathbb{N}}$ be a non-negative martingale that is adapted to $\{Y_n\}_{n\in\mathbb{N}}$ for which $\mathbb{E}M_0=1$. Let $T$ be a stopping time adapted to $\{Y_n\}_{n\in\mathbb{N}}$. Define a sequence of probability measures as $\mathbb{P}_n(A^\prime)=\mathbb{E}\mathbbm{1}_{A^\prime} M_n$, for $A^\prime\in\mathcal{F}$.
Then there exists a probability measure $\tilde{\mathbb{P}}$, such that $\tilde{\mathbb{P}}(A^{\prime})=\mathbb{P}_n(A^{\prime})$, for $A^\prime\in\mathcal{F}$ and $n\in\mathbb{N}$.
Furthermore, we have that $\mathbb{E}\mathbbm{1}_{\{T<\infty\}}=\tilde{\mathbb{E}}\mathbbm{1}_{\{T<\infty\}}M_T^{-1}$.
\end{lemma}
Our goal is to find a suitable martingale $M_n$ such that the strong efficiency criterion in \eqref{sec1eq3} is satisfied. Here we consider first a useful example proposed by \citeA{blanchetglynn2008}, where the authors develop an efficient state-dependent importance sampling strategy for estimating the tail probability of a random walk crossing a certain level. Before we go through the details of the example, we introduce the following definition.
\begin{defn}
A random variable $X$ is said to posses a long tail, if for every $c\in\mathbb{R}$, we have that $\mathbb{P}(X>t+c)\sim\mathbb{P}(X>t)$ as $t\to\infty$. $X$ is called subexponential if $\mathbb{P}(X_1^+ +X_2^+>t)\sim2\mathbb{P}(X^+>t)$ as $t\to\infty$, where $X_1$ and $X_2$ are independent copies of $X$. Moreover, $X$ is said to belong to the family $S^*$ if the following holds
$$2\mathbb{E}X^+\mathbb{P}(X>t)\sim\int_0^t\mathbb{P}(X>t-s)\mathbb{P}(X>s)ds$$
as $t\to\infty$.
\end{defn}
If $X_1$ belongs to $S^*$, then both the distribution of $X_1$ and its integrated tail are subexponential (cf.\ \citeNP[Theorem 3.2]{klueppelberg1988}) and, in particular, long tailed. Furthermore, the Pakes-Veraverbeke's theorem (cf.\ e.g.\ \citeNP{veraverbeke1977} and \citeNP{zachary2004}) says
\begin{equation}\label{sec2eq5}
\mathbb{P}\left(\max\limits_{n\geq0} S_n>x\right)\sim-\frac{1}{\mathbb{E}X_1}\int_x^\infty \mathbb{P}(X_1>t)dt,
\end{equation}
as $x\to\infty$, where $S_n=S_{n-1}+X_n$.
\begin{example}\label{sec2ex1}
Consider a random walk $\{S_n\}_{n\in\mathbb{N}}$ generated by a sequence of i.i.d.\ random variables $\{X_n\}_{n\in\mathbb{N}}$, i.e, $S_n=S_{n-1}+X_n$, $S_0=0$. Assume that $\mathbb{E}X_1<0$ and $X_1$ belongs to $S^*$. We are interested in estimating $\mathbb{P}\left(\tau(x)<\infty\right)=\mathbb{P}(\max_{n\geq0}S_n>x)$, where $\tau(x)=\inf\{n\geq0:S_n>x\}$.
Let $v(z)$ be a positive function on $(-\infty,x)$, let $P(y,dz)$ denote the transition kernel of the random walk. Instead of $P(y,dz)$, one can simulate the random walk via another transition kernel
\begin{equation}\label{sec2eq7}
Q(y,dz)=P(y,dz)\frac{v(z)}{w(y)},\quad\forall y\in(-\infty,x],\text{\ }z\in\mathbb{R},
\end{equation}
where $w(y)$ is the normalization constant that is given by $w(y)=\int_{\mathbb{R}}v(z)P(y,dz)$. Choosing
$$M_n^{-1}=\prod_{k=1}^n\frac{w(S_{k-1})}{v(S_k)}$$
and applying Lemma \ref{sec2lem1}, this yields a potential candidate of the estimator, which has the form $L(x)=\mathbbm{1}_{\{\tau(x)<\infty\}}M_{\tau(x)}^{-1}$. It is ``a potential candidate'' because for each $a\leq0$, any $M_n^{-1}$ constructed by the pair $w(\cdot+a)$ and $v(\cdot+a)$ is also a possible choice. Define a non-negative random variable $W$ that is independent of $\{X_n\}_{n\in\mathbb{N}}$ with tail probability
\begin{equation}\label{sec2eq4}
\mathbb{P}(W>t) \triangleq  \min\left[1,-\frac{1}{\mathbb{E}X_1}\int_t^\infty\mathbb{P}(X_1>s)ds\right].
\end{equation}
\citeA{blanchetglynn2008} propose to choose
\begin{equation}\label{sec2eq1}
v(z) \triangleq  \mathbb{P}(W>-(z-x)),
\end{equation}
and
\begin{equation}\label{sec2eq2}
w(y) \triangleq  \mathbb{P}(X_1+W>-(y-x)).
\end{equation}
From \eqref{sec2eq5} we can see that the choice above of $v(\cdot)$ and $w(\cdot)$ leads to a good approximation of the so-called zero-variance importance distribution, which involves sampling from the conditional distribution of the random walk given $\{\tau(x)<\infty\}$ (cf.\ \citeNP[Theorem 1]{blanchetglynn2008}). By showing
\begin{equation}\label{sec2eq6}
w(y)-v(y)=\omicron(\,\mathbb{P}(X_1>-y)\,),\quad \text{as\ }y\to-\infty, 
\end{equation}
(for details see \citeNP[Proposition 3]{blanchetglynn2008}) and, for each $\delta\in(0,1)$, the existence of a constant $a_*=a_*(\delta)\in(-\infty,0]$ such that
\begin{equation}\label{sec2eq3}
-\delta\leq\frac{v^2(y)-w^2(y)}{\mathbb{P}(X_1>-y)w(y)},\quad\forall y\leq x+a_*,
\end{equation}
the authors were able to control the second moment of the estimator via a Lyapunov bound (for details see \citeNP[Theorem 2, Proposition 2, Proposition 3]{blanchetglynn2008}). We summarize one of their results in the next theorem, which will prove to be useful in  our context.
\end{example}
\begin{thm}\cite[Theorem 3]{blanchetglynn2008}\label{sec2thm1}
Suppose that $\mathbb{E}X_1<0$ and $X_1$ belongs to $S^*$. Let $v$ and $w$ be defined as in \eqref{sec2eq1} and \eqref{sec2eq2}. For fixed $\delta\in(0,1)$, set $a_*=a_*(\delta)\leq0$ satisfying \eqref{sec2eq3}. Let an unbiased estimator of $\mathbbm{P}(\max_{n\in\mathbbm{N}}S_n>x)$ be given by
$$L_\tau(x)=\mathbbm{1}_{\{\tau(x)<\infty\}}\prod_{k=1}^{\tau(x)}\frac{w(S_{k-1}+a_*)}{v(S_k+a_*)}.$$
Then
$$\sup_{x>0}\frac{\mathbb{E}^{Q_{a_*}}L_\tau^2(x)}{\mathbb{P}\left(\max\limits_{n\geq0}S_n>x\right)^2}<\infty,$$
where $\mathbb{E}^{Q_{a_*}}$ denotes the expectation w.r.t.\ the random process $\{S_n\}_{n\in\mathbb{N}}$ having a one-step transition kernel
\begin{align*}
Q_{a_*}(y,dz)&=P(y,dz)\frac{v(z+a_*)}{w(y+a_*)}\\
&=\frac{\mathbb{P}(y+X_1\in dz)v(z+a_*)}{w(y+a_*)}\\
&=\mathbb{P}(y+X_1\in dz\,|\,X_1+W>-(y-x)-a_*).
\end{align*}
\end{thm}
Theorem \ref{sec2thm1} implies that the following algorithm is strongly efficient.
\begin{algm}\label{sec2algm1}
\begin{description}
\item[]
\begin{description}
\item[]
\item[STEP 0.] For fixed $\delta\in(0,1)$, set $a_*\longleftarrow a_*(\delta)\leq0$ satisfying \eqref{sec2eq3}.
\item[STEP 1.] Initialize $s\longleftarrow 0$ and $L\longleftarrow1$.
\item[STEP 2.] Set $s^\prime\longleftarrow s$, generate a random variable $Y$ with the distribution given by
$$\mathbb{P}(Y\in dz)=\mathbb{P}(s^\prime + X_1\in dz\,|\,X_1+W>-(s^\prime-x)-a_*),$$
where $W$ is defined as in \eqref{sec2eq4}. Update $s\longleftarrow s^\prime+Y$ and
$$L\longleftarrow\frac{w(s^\prime+a_*)}{v(s+a_*)}L.$$
\item[STEP 3.] If $s>x$ then return $L$. Otherwise, go to STEP 2.
\end{description}
\end{description}
\end{algm}

\section{Stochastic Perpetuity and Iterated Random Functions}\label{SectionSPIRF}
In this section we specify the random object that can be handled by our algorithm. We start with an example of stochastic perpetuity and construct a stochastic upper bound that can be written as a functional of a suitable random walk $S_n(\gamma)$, which will be defined in three different levels of generality---Example~\ref{sec3ex1}, general stochastic perpetuities, and stochastic recursions of the form (\ref{sec3eq7}). Furthermore, using the upper bound we can define crossing levels $s(x)$ (which  will be defined in three different levels of generality as well) and a stopping time
$$\tau_\gamma(x)=\inf\{n\geq0:S_n(\gamma)>s(x)\},$$
such that
\begin{equation}\label{sec3eq1}
\{Z>x\}\subseteq\{\max_{n\geq0}S_n(\gamma)>s(x)\}\text{\quad and\quad}\tau_\gamma(x)\leq T(x).
\end{equation}
Since the change of measure proposed in \citeA{blanchetglynn2008} is strongly efficient for estimating the tail probability of the maximum of heavy-tailed random walks, a natural strategy is to keep track of the random process $\{S_n(\gamma)\}_{n\in\mathbb{N}}$ while simulating $Z^{(m)}$, until the stopping time $\tau_\gamma(x)$. By doing this, we can construct a state-dependent change of measure using the path of the random walk until $\tau_\gamma(x)$ according to the method introduced in Example \ref{sec2ex1}. Then we simulate the path of the random walk after $\tau_\gamma(x)$ under the original measure. In the second step we extend the method to the general case. Other properties such as efficiency will be discussed in Section \ref{SectionAUBSE} and \ref{SectionAU}.
\subsection{Stochastic Perpetuity}\label{SectionSP}
To illustrate our idea, let us consider a simple example, namely, a stochastic perpetuity that generates exact one unit of money at each time point $n$. 

\begin{example}\label{sec3ex1}
Consider the Markov chain defined via the random difference equations
\begin{equation}\label{sec3eq0}
Z_{n+1} = A_{n+1} Z_n+1,\text{\ }t\in\mathbb{N},
\end{equation}
where $\mathbb{E}[\log A_1]<0$ and $\{A_n\}_{n\in\mathbb{N}}$ is an i.i.d.\ sequence of positive random variables, which is independent of $Z_0$. The unique stationary distribution of this Markov chain has the same distribution as the random variable
$$Z \triangleq  \sum_{n=0}^\infty e^{S_n},$$
where $S_n \triangleq  S_{n-1}+X_n$, $S_0 \triangleq  0$ and $X_n \triangleq  \log A_n$. Let $\gamma_1\in(0,-\mathbb{E}X_1)$ be fixed. For the stochastic perpetuity $Z$, we observe that
\begin{equation}\label{sec3eq2}
Z = \sum_{n=0}^\infty \exp(S_n)=\sum_{n=0}^\infty \exp\{S_n+n\gamma_1\}\exp(-n\gamma_1)\leq \exp\left\{\max\limits_{n\geq0}(S_n+n\gamma_1)\right\}\frac{1}{1-e^{-\gamma_1}}.
\end{equation}
Using \eqref{sec3eq2} we can define $s(x)\triangleq \log x+\log(1-e^{-\gamma_1})$ and $\tau_\gamma(x)\triangleq \inf\{n\geq0:S_n+n\gamma_1>s(x)\}$, such that \eqref{sec3eq1} holds. To see $\tau_\gamma(x)\leq T(x)$, suppose $\tau_\gamma(x)=N_\tau$, we have that $\max_{n\leq N_\tau-1}S_n(\gamma)\leq s(x)$. Combining this with the fact that
\begin{align*}
\sum_{n=0}^{N_\tau-1} \exp(S_n)&=\sum_{n=0}^{N_\tau-1} \exp\{S_n+n\gamma_1\}\exp(-n\gamma_1)\\
&\leq\exp\left\{\max\limits_{n\leq N_\tau-1}S_n+n\gamma_1\right\}\frac{1-e^{N_\tau}}{1-e^{-\gamma_1}}\\
&\leq\exp\left(s(x)\right)\frac{1}{1-e^{-\gamma_1}}\\
&=x,
\end{align*}
we can conclude that $T(x)\geq N_\tau=\tau_\gamma(x)$.
\end{example}

Example \ref{sec3ex1} shows that we can bound $Z$ by a functional of the mean-shifted random walk, if we assume that there is no randomness in $B_n$. However, extending the idea to general stochastic perpetuities is not straightforward. The reason is that we can not deal $B_n$ separately due to the potential dependence structure between $\{A_n\}_{n\in\mathbb{N}}$ and $\{B_n\}_{n\in\mathbb{N}}$. To be precise, consider the Markov chain defined via the stochastic difference equations
$$Z_{n+1} = A_{n+1} Z_n+B_{n+1},\text{\ }n\in\mathbb{N},$$
where $\mathbb{E}[\log A_1]<0$, $\mathbb{E}[\log^+ B_1]<\infty$ and $\{(A_n,B_n)\}_{n\in\mathbb{N}}$ is an i.i.d.\ sequence of positive random vectors that are independent of $Z_0$. First we need to assume the following.

\begin{assumption}\label{sec3ass1}
Assume that $(A_1,B_1)$ satisfies the conditions as follows.
\begin{description}
\item[a)] Let $A_1$, $B_1$>0 a.s, $\mathbb{E}[\log A_1]<0$ and $\mathbb{E}[\log^+B_1]<\infty$.
\item[b)] $\mathbb{E}[(\log^+(\max(A_1,B_1)))^{1+\eta}]<\infty$, for some $\eta>0$.
\item[c)] For $(A_1,B_1)$ we have that
$$\mathbb{P}(A_1>x,B_1\leq-x)=\omicron(\mathbb{P}(\max(A_1,B_1)>x)).$$
\end{description}
\end{assumption}

Under Assumption \ref{sec3ass1}, it is well known (cf.\ \citeNP{buraczewskidamekmikosch2016} and \citeNP{dyszewski2016}) that the unique stationary distribution of this Markov chain exists, has right-unbounded support and has the same distribution as the random variable
$$Z \triangleq  \sum_{n=0}^\infty B_{n+1} e^{S_n},$$
where $S_n \triangleq  S_{n-1}+X_n$, $S_0 \triangleq  0$ and $X_n \triangleq  \log A_n$. In the next step we want to construct an upper bound for $Z$, which can be written as a functional of a suitable random walk $S_n(\gamma)$. We have the following Lemma.
\begin{lemma}\label{sec3lem1}
Under Assumption \ref{sec3ass1}, there exists a constant $\gamma_{2}$ such that
$$\mathbb{E}[\max\left(\log^+ B_1-\gamma_{2},\log A_1\right)]<0.$$
Moreover, there exists a constant $\gamma_{1}\in(0,-\mathbb{E}\max(\log^+B_1-\gamma_{2},\log A_1))$ such that
\begin{equation}\label{sec3eq3}
Z \leq e^{\gamma_{2}}\sum_{n=0}^\infty e^{S^\prime_n}\leq \exp\left\{\max\limits_{n\geq0}S_n(\gamma)\right\}\frac{e^{\gamma_{2}}}{1-e^{-\gamma_{1}}},
\end{equation}
where $\gamma=(\gamma_{1},\gamma_{2})$, $S^\prime_n=S^\prime_{n-1}+\max(\log^+B_n-\gamma_{2},\log A_n)$ and $S_n(\gamma)=S^\prime_n+n\gamma_{1}$.
\end{lemma}
\begin{proof}
Under Assumption \ref{sec3ass1} a) we obtain that
$$\lim_{\gamma^\prime_2\to\infty}\max(\log^+ B_1-\gamma^\prime_2,\log A_1)=\log A_1,\text{\ a.s,}$$
since $\log^+ B_1$ and $\log A_1$ are finite a.s. Furthermore, we have that
\begin{equation}\label{sec3eq4}
|\max(\log^+ B_1-\gamma^\prime_2,\log A_1)|\leq|\max(\log^+ B_1,\log A_1)|+|\log A_1|.
\end{equation}
Using Assumption \ref{sec3ass1} b) we have that $\mathbb{E}\log^+ A_1 \leq \mathbb{E}\log^+ \max\left(A_1,B_1\right)<\infty$. Since $\mathbb{E}\log A_1<0$, we conclude that $\mathbb{E}\left|\log A_1\right|<\infty$.
Combining this with \eqref{sec3eq4} we have that for any $\gamma_1>0$,
\begin{align*}
\mathbb{E}\left|\max(\log^+ B_1-\gamma^\prime_1,\log A_1)\right|&\leq\mathbb{E}\left|\max(\log^+ B_1,\log A_1)\right|+\mathbb{E}\left|\log A_1\right|\\
&=\mathbb{E}\left|\log^+(\max(B_1,A_1))\right|+\mathbb{E}\left|\log A_1\right|<\infty.
\end{align*}
Using the dominated convergence theorem we obtain the existence of $\gamma_{2}$. For the stochastic perpetuity $Z$ and the constant $\gamma_{2}>0$, we have that
\begin{equation}
Z\leq\sum_{n=0}^\infty\max\left(B_{n+1},1\right)e^{S_n}=e^{\gamma_{2}}\sum_{n=0}^\infty e^{(\log^+ B_{n+1}-\gamma_{2})+S_n}\leq e^{\gamma_{2}}\sum_{n=0}^\infty e^{S^\prime_n},\label{sec3eq5}
\end{equation}
where $S^\prime_n=S^\prime_{n-1}+\max(\log^+B_n-\gamma_{2},\log A_n)$. To see the last inequality, comparing $S^\prime_{n+1}$ with $(\log^+B_{n+1}-\gamma_{2})+S_n$ component wise, we have that
\begin{align*}
\left(\log^+B_{n+1}-\gamma_{2}\right)+S_n&= \left(\log^+B_{n+1}-\gamma_{2}\right)+\sum_{k=1}^n \log A_k\\
&\leq \max(\log^+B_{n+1}-\gamma_{2},\log A_{n+1})+\sum_{k=1}^n \max(\log^+B_k-\gamma_{2},\log A_k)= S^\prime_{n+1}.
\end{align*}
Now let $\gamma_{1}\in(0,-\mathbb{E}\max(\log^+B_1-\gamma_{2},\log A_1))$ be fixed. From \eqref{sec3eq5}, we observe that
$$Z \leq e^{\gamma_{2}}\sum_{n=0}^\infty e^{S^\prime_n}=e^{\gamma_{2}}\sum_{n=0}^\infty \exp\{S^\prime_n+n\gamma_{1}\}\exp(-n\gamma_{1})\leq \exp\left\{\max\limits_{n\geq0}S_n(\gamma)\right\}\frac{e^{\gamma_{2}}}{1-e^{-\gamma_{1}}},$$
where $\gamma=(\gamma_{1},\gamma_{2})$ and $S_n(\gamma)\triangleq S_{n-1}(\gamma)+\max(\log^+B_n-\gamma_{2},\log A_n)+\gamma_1=S^\prime_n+n\gamma_{1}$.
\end{proof}

Now from \eqref{sec3eq3} we can define $s(x)\triangleq \log x-\gamma_{2}+\log(1-e^{-\gamma_{1}})$ and $\tau_\gamma(x)\triangleq \inf\{n\geq0:S_n(\gamma)>s(x)\}$, such that \eqref{sec3eq1} holds.

\subsection{Iterated Random Functions}\label{SectionIRF}
As we indicated in the introduction, stochastic perpetuities can be considered as a special case of \eqref{sec1eq1} with $\Psi_n$ being an affine transformation. Thus, it is natural to consider the Markov chain given by \eqref{sec1eq1}, where for each $\Psi_n$ there exists a random vector $(A_n,B_n,D_n)$ satisfying 
\begin{equation}\label{sec3eq7}
A_n z+B_n-D_n\leq\Psi_n(t)\leq A_n z^+ +B_n^+ +D_n,
\end{equation}
for $z\in\mathbb{R}$. We assume that $\{(A_n,B_n,D_n)\}_{n\in\mathbb{N}}$ are i.i.d.\ and $\{\Psi_n\}_{n\in\mathbb{N}}$ is a sequence of i.i.d.\ positive random Lipschitz functions with
$$\text{Lip}(\Psi_n) \triangleq  \sup_{t_1\neq t_2}\left|\frac{\Psi_n(t_1)-\Psi_n(t_2)}{t_1-t_2}\right|.$$
Analogous to Assumption \ref{sec3ass1}, we assume the following holds.
\begin{assumption}\label{sec3ass2}
Assume that $\Psi_n$ is such that $Z^{(n)}$ defined as in \eqref{sec1eq5} is increasing in $n$. Moreover, assume that \eqref{sec3eq7} hold, and that $(A_1,B_1,D_1)$ satisfies the following conditions:
\begin{description}
\item[a)] $A_1,D_1>0$ a.s, $\mathbb{E}[\log A_1]>-\infty$ and $\mathbb{E}[\log\text{Lip}(\Psi_1)]<0$. Moreover, $\mathbb{E}\left[\log^+|B_1^+ +D_1|\right]<\infty$ and $\mathbb{E}\left[\log^+|B_1-D_1|\right]<\infty$.
\item[b)] $\mathbb{E}\left[(\log^+(\max(A_1,B_1)))^{1+\eta}\right]<\infty$, for some $\eta>0$.
\item[c)] The following tail behaviors
$$\mathbb{P}(\max(A_1,B_1^+ +D_1)>x)\sim\mathbb{P}(\max(A_1,B_1)>x),$$
$$\mathbb{P}(\max(A_1,B_1-D_1)>x)\sim\mathbb{P}(\max(A_1,B_1)>x)$$
and
$$\mathbb{P}(A_1>x,B_1-D_1\leq-x)=\omicron(\mathbb{P}(\max(A_1,B_1)>x))$$
hold.
\end{description}
\end{assumption}

\begin{rmk}
We want to mention that there are interesting examples that satisfy the increasing property of the backward iteration. For instance, consider the stochastic equation given by
$$Z_{n+1}=\sqrt{A_{n+1}(Z_n)^2+B_{n+1}Z_n+C_{n+1}}.$$
This corresponds to a second order random polynomial equation, which is studied by \citeA{goldie1991}.

\end{rmk}

Define $Z \triangleq  \lim_{n\to\infty}\Psi_{1:n}(0)$. Recall that (cf.\ \citeNP{dyszewski2016}), under Assumption \ref{sec3ass2}, the unique stationary solution to \eqref{sec1eq1} exists, has the same distribution as $Z$, has right-unbounded support and can be bounded from above with a stochastic perpetuity $\bar{Z}$, which is given by
$$\bar{Z} \triangleq  \sum_{n=0}^\infty \bar{B}_{n+1} e^{S_n},$$
where $\bar{B}_n \triangleq  \max\left(B_n^+ +D_n,1\right)$, $S_n \triangleq  S_{n-1}+ \log(A_n)$ and $S_0 \triangleq  0$ (cf.\ \citeNP{dyszewski2016}). Analogous to the previous section, our goal should be to construct an upper bound for $\bar{Z}$ (and thus for $Z$) that can be written as a functional of the maximum of a suitable random walk $S_n(\gamma)$. First we claim the following lemma.
\begin{lemma}\label{sec3lem2}
Under Assumption \ref{sec3ass2}, there exists a constant $\gamma_{2}$ such that
$$\mathbb{E}\left[\max\left(\log^+ \left(B_1^+ + D_1\right)-\gamma_{2},\log A_1\right)\right]<0.$$
Moreover, there exists a constant $\gamma_{1}\in(0,-\mathbb{E}\max(\log^+B_1-\gamma_{2},\log A_1))$ such that
\begin{equation}\label{sec3eq6}
Z \leq e^{\gamma_{2}}\sum_{n=0}^\infty e^{S^\prime_n}\leq \exp\left\{\max\limits_{n\geq0}S_n(\gamma)\right\}\frac{e^{\gamma_{2}}}{1-e^{-\gamma_{1}}},
\end{equation}
where $\gamma=(\gamma_{1},\gamma_{2})$, $S^\prime_n=S^\prime_{n-1}+\max\left(\log^+ \left(B_n^+ + D_n\right)-\gamma_{2},\log A_n\right)$ and $S_n(\gamma)=S^\prime_n+n\gamma_{1}$.
\end{lemma}
\begin{proof}
Analogous to the proof of Lemma \ref{sec3lem1}.
\end{proof}

Now from \eqref{sec3eq6} we can define $s(x) \triangleq  \log x-\gamma_{2}+\log(1-e^{-\gamma_{1}})$ and $\tau_\gamma(x) \triangleq  \inf\{n\geq0:S_n(\gamma)>s(x)\}$, such that \eqref{sec3eq1} holds.

\section{Asymptotic Upper Bound and Strong Efficiency}\label{SectionAUBSE}
Recall that in Section \ref{SectionSPIRF} we have developed a stochastic upper bound for each of our random objects. Moreover, using these upper bounds we also have defined a crossing level $s(x)$ and a stopping time
$$\tau_\gamma(x)=\inf\{n\geq0:S_n(\gamma)>s(x)\},$$
such that $\{Z>x\}\subseteq\{\max_{n\geq0}S_n(\gamma)>s(x)\}$ and $\tau_\gamma(x)\leq T(x)$, for each case respectively.

It turns out that the upper bounds we derived in the previous section are not only helpful for constructing our algorithm, they can also be used to derive an asymptotic upper bound for $\mathbb{P}(Z>x)$ as $x\to\infty$. In this section, we first analyze the asymptotic behavior of $\mathbb{P}(Z>x)$. After that, we show that our estimator is strongly efficient.
\subsection{Asymptotic Upper Bound}
In the theory of large deviations, one is interested in results such as the asymptotic behavior of $\mathbb{P}(Z>x)$ as $x\to\infty$. More precisely, we are interested in finding a function $f(x)$ such that $\mathbb{P}(Z>x)\sim f(x)$ as $x\to\infty$. Usually, obtaining an asymptotic lower bound is easier than obtaining an upper bound. It turns out that the stochastic upper bounds we derived for $Z$ can be used in deriving the asymptotic upper bound of $\mathbb{P}(Z>x)$. To illustrate this, let us get back to our example.

\addtocounter{example}{-1}
\begin{example}[continued]
Consider a stochastic perpetuity with $B_n=1$, i.e.\ $Z = \sum_{n=0}^\infty e^{S_n}$. Moreover, assume that the integrated tail of $X_1$ is subexponential. On the one hand, using \eqref{sec3eq5}, we obtain that
$$\mathbb{P}(Z>x)\leq\mathbb{P}\left(\max_{n\geq0}S_n(\gamma)>s(x)\right).$$
Since the integrated tail of $X_1+\gamma_1$ is also subexponential, applying the Pakes-Veraverbekes Theorem we have that
\begin{equation}\label{sec4eq1}
\limsup_{x\to\infty}\frac{\mathbb{P}(Z>x)}{\int_{\log x}^\infty \mathbb{P}(X_1 > t) dt}\leq\limsup_{x\to\infty}\frac{\mathbb{P}\left(\max\limits_{n\geq0}S_n(\gamma)>s(x)\right)}{\int_{\log x}^\infty \mathbb{P}(X_1 > t) dt}=-\frac{1}{\mathbb{E}X_1+\gamma_1}.
\end{equation}
Letting $\gamma_1\to0$ we conclude that
$$\limsup_{x\to\infty}\frac{\mathbb{P}(Z>x)}{\int_{\log x}^\infty \mathbb{P}(X_1 > t) dt}\leq-\frac{1}{\mathbb{E}X_1}.$$
On the other hand we observe that $Z\geq \exp\left(\max_{n\geq0} S_n\right)$. Therefore applying again the Pakes-Veraverbekes Theorem we obtain that
\begin{equation}\label{sec4eq2}
\liminf_{x\to\infty}\frac{\mathbb{P}(Z>x)}{\int_{\log x}^\infty \mathbb{P}(X_1 > t) dt}\geq\liminf_{x\to\infty}\frac{\mathbb{P}\left(\max\limits_{n\geq0}S_n>\log x\right)}{\int_{\log x}^\infty \mathbb{P}(X_1 > t) dt}=-\frac{1}{\mathbb{E}X_1}.
\end{equation}
Combining \eqref{sec4eq1} and \eqref{sec4eq2} we conclude that
$$\mathbb{P}(Z>x)\sim-\frac{1}{\mathbb{E}X_1}\int_{\log x}^\infty \mathbb{P}\left(X_1>t\right) dt.$$
\end{example}

Consider the stationary distribution of the Markov chain given by \eqref{sec1eq1}. As we indicated earlier, \citeA{dyszewski2016} shows that under subexponential assumptions on the random variable $\log^+\max(A_1,B_1)$ the tail asymptotics can be described using the integrated tail function of $\log^+\max(A_1,B_1)$. However, the upper bound we derived in Section \ref{SectionIRF} yields us a shorter proof for the asymptotic upper bound in \citeA[Theorem 3.1]{dyszewski2016}.
\begin{lemma}{\rm{(A shorter proof of the asymptotic upper bound in \citeNP[Theorem 3.1]{dyszewski2016})}}\label{sec4lem1}
Let Assumption \ref{sec3ass2} hold. Furthermore, assume that the integrated tail of $\log\max(A_1,B_1)$ is subexponential. Then we have that
$$\limsup_{x\to\infty}\frac{\mathbb{P}(Z>x)}{\bar{F}_I(\log(x))}\leq-\frac{1}{\mathbb{E}[\log(A_1)]},$$
where $\bar{F}_I$ denotes the integrated tail of $\log\max(A_1,B_1)$.
\end{lemma}
\begin{proof}
From the upper bound we constructed in Section \ref{SectionIRF}, we know that
\begin{equation}\label{sec4eq3}
\mathbb{P}(Z>x)\leq\mathbb{P}\left(\max_{n\geq0}S_n(\gamma)>s(x)\right).
\end{equation}
Due to Assumption \ref{sec3ass2} c) we know that the integrated tail of $\log(\max(A_1,B_1^+ +D_1,1))$ is also subexponential. Moreover, we have the following inequality:
\begin{align*}
\log\left(\max\left(A_1,B_1^+ +D_1,1\right)\right)-\gamma_{2}&\leq\log\left(\max\left(A_1,e^{-\gamma_{2}}\left(B_1^+ +D_1\right),e^{-\gamma_2}\right)\right)\leq\log\left(\max\left(A_1,B_1^+ +D_1,1\right)\right).
\end{align*}
The increments of the random walk $S_n(\gamma)$ have a subexponential integrated tail. Using the Pakes-Veraverbeke theorem we get the following relationship for the RHS of \eqref{sec4eq3}, namely
\begin{equation}\label{sec4eq7}
\mathbb{P}\left(\max_{n\geq0}S_n(\gamma)>s(x)\right)\sim-\frac{1}{\mathbb{E}[\max(\bar{B}_1-\gamma_{2},\log A_1)]+\gamma_{1}}\bar{F}_I(\log(x)).
\end{equation}
Now, letting $\gamma_{2}\to\infty$ and $\gamma_{1}\to0$ yields the result.
\end{proof}

\begin{rmk}
Note that the assumption of an increasing backward iteration of $\Psi_n$ is not needed in the proof of Lemma \ref{sec4lem1}. Therefore, it provides indeed a shorter proof of the asymptotic upper bound given in \citeA[Theorem 3.1]{dyszewski2016}.
\end{rmk}

\subsection{Strong Efficiency}
Given the asymptotic behavior of $\mathbb{P}(Z>x)$, we are able to show the strong efficiency of our estimator. Recall that, based on \eqref{sec3eq1}, our algorithm constructs a state-dependent change of measure according to the methods introduced by \citeA{blanchetglynn2008}. Define the following elements, which are needed for the change of measure, via
$$\mathbb{P}(W_\gamma>t) \triangleq  \min\left[1\,,\,\frac{1}{\mathbb{E}S_1(\gamma)}\int_t^\infty\mathbb{P}(S_1(\gamma)>s)ds\right],$$
\begin{equation}\label{sec4eq4}
v_\gamma(z) \triangleq  \mathbb{P}(W_\gamma>-(z-s(x))),
\end{equation}
and
\begin{equation}\label{sec4eq5}
w_\gamma(y) \triangleq  \mathbb{P}(S_1(\gamma)+W_\gamma>-(y-s(x))).
\end{equation}
We propose an estimator and show its strong efficiency in the following theorem.
\begin{thm}\label{sec4thm1}
Suppose that $\mathbb{E}S_1(\gamma)<0$ and $S_1(\gamma)$ belongs to $S^*$. Let $v_\gamma$ and $w_\gamma$ be defined as in \eqref{sec4eq4} and \eqref{sec4eq5}. For fixed $\delta\in(0,1)$, one can choose $a_*=a_*(\delta)\leq0$ so that
$$-\delta\leq\frac{v_\gamma^2(y)-w_\gamma^2(y)}{\mathbb{P}(X_1>-y)w_\gamma(y)},\quad \forall y\leq s(x)+a_*.$$
Let 
\begin{equation}\label{sec4eq6}
L_T(x)\triangleq \mathbbm{1}_{\{T(x)<\infty\}}\prod_{k=1}^{\tau_\gamma(x)}\frac{w_\gamma(S_{k-1}(\gamma)+a_*)}{v_\gamma(S_k(\gamma)+a_*)}.
\end{equation}
Then $L_T(x)$ is an unbiased estimator of $\mathbb{P}(Z>x)$ and
$$\sup_{x>1}\frac{\mathbb{E}^{Q^\gamma_{a_*}}L_T^2(x)}{\mathbb{P}(Z>x)^2}<\infty,$$
where $\mathbb{E}^{Q^\gamma_{a_*}}$ denotes the expectation w.r.t.\ the Markov chain $\{S_n(\gamma)\}_{n\in\mathbb{N}}$ having a one-step transition kernel
$$Q^\gamma_{a_*}(y,dz)=P(y,dz)\frac{v_\gamma(z+a_*)}{w_\gamma(y+a_*)}.$$
\end{thm}
\begin{proof}
Let
$$M_n^{-1}=\prod_{k=1}^{n}\frac{w_\gamma(S_{k-1}(\gamma)+a_*)}{v_\gamma(S_k(\gamma)+a_*)}.$$
Obviously, $\{M_n\}_{n\in\mathbb{N}}$ is a martingale, and therefore, $\{M_{n\wedge\tau_\gamma(x)}\}_{n\in\mathbb{N}}$ is also a martingale. Since $\tau_\gamma(x)\leq T(x)$, applying Lemma \ref{sec2lem1}, we can conclude that
$$\mathbb{E}^{Q^\gamma_{a_*}}L_T(x)=\mathbb{P}(T(x)<\infty)=\mathbb{P}(Z>x).$$
For the strong efficiency we have that
\begin{align*}
\frac{\mathbb{E}^{Q^\gamma_{a_*}}L_T^2(x)}{\mathbb{P}(Z>x)^2} &= \frac{\mathbb{E}^{Q^\gamma_{a_*}}\left[\mathbbm{1}_{\{Z>x\}}M_{\tau_\gamma}^{-2}(x)\right]}{\mathbb{P}(Z>x)^2}\\
&\leq \frac{\mathbb{E}^{Q^\gamma_{a_*}}\left[\mathbbm{1}_{\{\max\limits_{n\geq0}S_n(\gamma)>s(x)\}}M_{\tau_\gamma}^{-2}(x)\right]}{\mathbb{P}(Z>x)^2}\\
&= \frac{\mathbb{E}^{Q^\gamma_{a_*}}\left[\mathbbm{1}_{\{\max\limits_{n\geq0}S_n(\gamma)>s(x)\}}M_{\tau_\gamma}^{-2}(x)\right]}{\mathbb{P}\left(\max\limits_{n\geq0}S_n(\gamma)>s(x)\right)^2}\Bigg(\,\frac{\mathbb{P}\left(\max\limits_{n\geq0}S_n(\gamma)>s(x)\right)}{\mathbb{P}\left(Z>x\right)}\,\Bigg)^2,
\end{align*}
where the first term in the last equation is guaranteed to be bounded over $x\in(1,\infty)$ due to Theorem \ref{sec2thm1}. Hence, only the latter term remains to be analyzed. Define
$$\chi(x)=\frac{\mathbb{P}\left(\max\limits_{n\geq0}S_n(\gamma)>s(x)\right)}{\mathbb{P}\left(Z>x\right)}.$$
From \citeA[Theorem 3.1]{dyszewski2016} we have that
\begin{equation}\label{sec4eq8}
\liminf_{x\to\infty}\frac{\mathbb{P}(Z>x)}{\bar{F}_I(\log(x))}\geq-\frac{1}{\mathbb{E}\log A_1}.
\end{equation}
Since by assumption the integrated tail $\bar{F}_I$ is subexponential, it is in particular long tailed. Combining \eqref{sec4eq7} and \eqref{sec4eq8} we obtain that
\begin{equation}
\limsup_{x\to\infty}\chi(x)\leq\frac{\mathbb{E}\log A_1}{\mathbb{E}[\max(\bar{B}_1-\gamma_{2},\log A_1)]+\gamma_{1}}.
\end{equation}
Using the fact that $\chi(x)$ is bounded over a compact interval, we obtain that
$$\sup_{x>1}\frac{\mathbb{P}\left(\max\limits_{n\geq0}S_n(\gamma)>s(x)\right)}{\mathbb{P}\left(Z>x\right)}<\infty.$$
\end{proof}

\section{Asymptotic Unbiasedness}\label{SectionAU}
The estimator derived in Theorem \ref{sec4thm1} requires the computation of $\mathbbm{1}_{\{Z>x\}}$, and hence, is unbiased only if we can generate $Z$ in finite time. Generating a perfect sample from $Z$ in our current setting is not straightforward, although there is plenty of literature on this topic; see, for example, \citeA{blanchetwallwater2015} and \citeA{blanchetsigman2011}. Conditional on $\{\tau_\gamma(x)<\infty\}$, using the strong Markov property, we have that
$$Z = \Psi_{1:\tau_\gamma(x)}\left(Z^\prime\right),$$
where $Z^\prime \triangleq  \lim_{M\to\infty} \Psi_{\tau_\gamma(x)+1:\tau_\gamma(x)+M}(0) \,{\buildrel d \over =}\, Z$ is independent of $\Psi_{1:\tau_\gamma(x)}$. Therefore, a natural choice for approximating the distribution of $Z^\prime$ is a truncated sum. More precisely, letting $M\in\mathbb{N}$ be fixed; our modified estimator takes the form
$$L^\Delta_T(x,M)=\mathbbm{1}_{\{\tau_\gamma(x)<\infty,\, \Psi_{1:\tau_\gamma(x)}\left({Z^\prime}^{(M)}\right)>x\}}\prod_{k=1}^{\tau_\gamma(x)}\frac{w_\gamma(S_{k-1}(\gamma)+a_*)}{v_\gamma(S_k(\gamma)+a_*)},$$
where ${Z^\prime}^{(M)} \triangleq  \Psi_{\tau_\gamma(x)+1:\tau_\gamma(x)+M}(0)$. To illustrate this, let us consider an extension of Example \ref{sec3ex1}. 

\addtocounter{example}{-1}
\begin{example}[continued]
Consider a stochastic perpetuity with $B_n=1$. Moreover, for $\alpha>0$, assume that $X_1$ is regularly varying with index $\alpha+1$; i.e, for the tail distribution of $X_1$, we have that $\bar{F}(x)\sim x^{-\alpha-1}L(x)$ with $L$ being a slowly varying function. Let $a(x)$ denote the auxiliary function of $X_1$ (c.f, e.g.\ \citeNP{asmussenklueppelberg1996}) that is given by $a(x)=x/\alpha$ in the regularly varying case. On the set $\{\tau_\gamma(x)<\infty\}$, using the strong Markov property, we have that
$$Z = \sum_{n=0}^\infty e^{S_n}=A^\prime_x Z^\prime+B^\prime_x,$$
where $A^\prime_x=e^{S_{\tau_\gamma(x)}}$, $B^\prime_x=\sum_{n=0}^{\tau_\gamma(x)-1}e^{S_n}$, and $Z^\prime=\sum_{n=0}^{\infty} e^{S_{\tau_\gamma(x)+n}-S_{\tau_\gamma(x)}}$ is a random variable that is independent of $(A^\prime_x,B^\prime_x)$. Our modified estimator takes the form
$$L^\Delta_T(x,M)=\mathbbm{1}_{\{\tau_\gamma(x)<\infty,\, A^\prime_x {Z^\prime}^{(M)}+B^\prime_x>x\}}\prod_{k=1}^{\tau_\gamma(x)}\frac{w_\gamma(S_{k-1}(\gamma)+a_*)}{v_\gamma(S_k(\gamma)+a_*)},$$
where ${Z^\prime}^{(M)} = \sum_{n=0}^Me^{S_{\tau_\gamma(x)+n}-S_{\tau_\gamma(x)}}$. 
The relative bias is defined as
$$\Delta(x,M)=\Bigg|\, \frac{\mathbb{E}^{Q^\gamma_{a_*}}L^\Delta_T(x,M)-\mathbb{P}(T(x)<\infty)}{\mathbb{P}(T(x)<\infty)}\, \Bigg|.$$
From $\{\tau_\gamma(x)<\infty,\,\sum_{n=0}^{\tau_\gamma(x)+M}e^{S_n}>x\}\subseteq\{T(x)<\infty\}\subseteq\{\tau_\gamma(x)<\infty\}$ we obtain that
\begin{equation}\label{sec5eq1}
\Delta(x,M)=\Theta(x,M)\frac{\mathbb{P}\left(\max\limits_{n\geq0}S_n(\gamma)>s(x)\right)}{\mathbb{P}(Z>x)},
\end{equation}
where
\begin{align*}
\Theta(x,M)&=\mathbb{P}\left(A^\prime_x{Z^\prime}^{(M)}+B^\prime_x\leq x,A^\prime_xZ^\prime+B^\prime_x>x\, \middle|\, \tau_\gamma(x)<\infty\right)\\
&=\mathbb{P}\left(\frac{\log{{Z^\prime}^{(M)}}}{a(\log x)}\leq \frac{\log{\left(x-B^\prime_x\right)}}{a(\log x)}-\frac{\log{\left(A^\prime_x\right)}}{a(\log x)}<\frac{\log{Z^\prime}}{a(\log x)} \,\middle|\, \tau_\gamma(x)<\infty\right)\\
&=\mathbb{P}\left(\frac{\log{{Z^\prime}^{(M)}}}{a(\log x)}\leq \frac{\log{\left(x-B^\prime_x\right)}}{a(\log x)}-\frac{S_{\tau_\gamma(x)}(\gamma)-\tau_\gamma(x)\gamma}{a(\log x)}<\frac{\log{Z^\prime}}{a(\log x)} \,\middle|\, \tau_\gamma(x)<\infty\right).
\end{align*}
Note that we have seen in the proof of Theorem \ref{sec4thm1},  the latter term in the RHS of \eqref{sec5eq1} is bounded. Therefore, to show that the relative bias vanishes as $x \to \infty$, it is enough to  show that $\Theta(x,M)\to 0$ as $x\to \infty$. In the following corollary we derive the limiting distribution of $\xi_x$ conditional on $\{\tau_\gamma(x)<\infty\}$, where
$$\xi_x \triangleq  \frac{\log{\left(x-B^\prime_x\right)}}{a(\log x)}-\frac{S_{\tau_\gamma(x)}(\gamma)-\tau_\gamma(x)\gamma}{a(\log x)}=\frac{\log{\left(1-\frac{B^\prime_x}{x}\right)}}{a(\log x)}-\frac{S_{\tau_\gamma(x)}(\gamma)-\log x}{a(\log x)}+\frac{\tau_\gamma(x)\gamma}{a(\log x)}.$$
\begin{prop}\label{sec5corol1}
Let $\mu\triangleq-\mathbb{E}X_1$. Conditional on $\{\tau_\gamma(x)<\infty\}$, $\xi_x$ converges in distribution to $\xi \,{\buildrel d \over =}\, \gamma V_\alpha/(\mu-\gamma)-T_\alpha$ as $x\to\infty$, where $V_\alpha$ is a positive random variable and its tail is given by
$$\bar{G}_\alpha(x)=\mathbb{P}(V_\alpha>x)=(1+x/\alpha)^{-\alpha},\text{\ } x>0,$$
and $T_\alpha$ is defined on the same probability space, such that $\mathbb{P}(V_\alpha>x,T_\alpha>y)=\bar{G}_\alpha(x+y)$. Moreover, the density of $\xi$ is given by
$$f_\xi(y)=\frac{\mu-\gamma}{\mu}\left(1-\frac{y}{\alpha}\right)^{-\alpha-1}\mathbbm{1}_{\{y<0\}}+\frac{\mu-\gamma}{\mu}\left(1+\frac{(\mu-\gamma)y}{\alpha\gamma}\right)^{-\alpha-1}\mathbbm{1}_{\{y\geq0\}}.$$
\end{prop}

Since both random variables $\log{{Z^\prime}^{(M)}}/a(\log x)$ and $\log{Z^{\prime}}/a(\log x)$ converge in probability to $0$ as $x\to\infty$, the relative bias vanishes.
\end{example}

It should be noted that the result from the example above remains valid, if we assume that $X_1$ belongs to the maximum domain of attraction of the Gumbel distribution. However, we want to show the result of vanishing relative bias in a general context as described in Section \ref{SectionIRF}. Moreover, we are only assuming that the integrated tail of $\log(\max(A_n,B_n))$ is subexponential. We need the following Lemma, of which the proof uses a similar technique as the proof of Theorem 1 in \citeA{palmowskizwart2007}.
\begin{lemma}
Let $\mu\triangleq-\mathbb{E}X_1$ and $\mu_\gamma\triangleq-\mathbb{E}S_1(\gamma)$. For $\nu,K>0$ consider the sets
$$E^{(1)}_n=E^{(1)}_n(K,\nu)=\left\{S_j\in(-j\left(\mu+\nu)-K,-j(\mu-\nu)+K\right),\,j\leq n\right\},$$
$$E^{(2)}_n=E^{(2)}_n(K,\nu)=\left\{S_j(\gamma)\in\left(-j(\mu_\gamma+\nu)-K,-j(\mu_\gamma-\nu)+K\right),\,j\leq n\right\},$$
and
$$E^{(3)}_n=E^{(3)}_n(K,\nu)=\left\{|\underline{B}_j|\leq e^{\nu j+K},\,j\leq n\right\},$$
where $\underline{B}_j=B_j-D_j$. Then, for $\nu,\epsilon>0$, there exists $K>0$, such that
$$\mathbb{P}\left(\bigcap_{n\geq1}\left(E^{(1)}_n\cap E^{(2)}_n\cap E^{(3)}_n\right)\right)\geq1-\epsilon.$$
\end{lemma}
\begin{proof}
In the proof of Theorem 1 in \citeA{palmowskizwart2007}, the authors state that for any $\nu>0$ and any i.i.d.\ sequence $\{Y_n\}_{n\geq0}$ with $\mathbb{E}\left[\log^+|Y_1|\right]<\infty$, it holds that
$$\mathbb{P}\left(|Y_j|\leq e^{\nu j+K},\,j\leq n\right)\to1,$$
as $K\to\infty$ uniformly with respect to $n$. Using this argument we conclude that $\mathbb{P}\left(E^{(3)}_n\right)\to1$ as $K\to\infty$ uniformly with respect to $n$. Further, combining this fact with the SLLN for $\{S_n\}_{n\geq0}$ and $\{S_n(\gamma)\}_{n\geq0}$ (for details see eg.\ \citeNP[Lemma 3.1]{asmussenschmidlischmidt1999}), we can always take $K$ large enough such that
$$\mathbb{P}\left(E^{(1)}_n\cap E^{(2)}_n\cap E^{(3)}_n\right)\geq 1-\epsilon,$$
for all $n\in\mathbb{N}$. Finally, since the sequence of sets $\left\{E^{(1)}_n\cap E^{(2)}_n\cap E^{(3)}_n\right\}_{n\geq0}$ is decreasing in the sense of inclusion, we obtain the result.
\end{proof}
Due to the fact that $\left\{\tau_\gamma(x)<\infty,\Psi_{1:\tau_\gamma(x)}\left({Z^\prime}^{(M)}\right)>x\right\}\subseteq\{T(x)<\infty\}$, in order to prove the vanishing relative bias result, it is sufficient to show that
\begin{equation}\label{SectionGCeq1}
\liminf_{x\to\infty}\frac{\mathbb{P}\left(\tau_\gamma(x)<\infty,\Psi_{1:\tau_\gamma(x)}\left({Z^\prime}^{(M)}\right)>x\right)}{\mathbb{P}\left(T(x)<\infty\right)}\geq1.
\end{equation}
In the following theorem we use a similar proof technique as in the proof of Theorem 3.1 in \citeA{dyszewski2016}.
\begin{thm}\label{sec5thm1}
Let Assumption 2 hold. Moreover, we are assuming that the integrated tail of $\log(\max(A_1,B_1))$ is subexponential. Then \eqref{SectionGCeq1} holds.
\end{thm}
\begin{proof}
Define
\begin{align*}
E_n=E^{(1)}_n\cap E^{(2)}_n\cap E^{(3)}_n\cap\left\{\max\left(A_{n+1},\underline{B}_{n+1}\right)>xe^{n(\mu+\nu)+L+K},\underline{B}_{n+1}\geq -xe^{n(\mu-\nu)-K}\right\}\cap\{{Z^\prime}^{(1)}>\nu\},
\end{align*}
where $L>0$ is chosen to be large enough, such that the sets $\{E_n\}_{n\geq0}$ are disjoint. Moreover, we can show that $E_n\subseteq\{\tau_\gamma(x)=n+1,\Psi_{1:\tau_\gamma(x)}({Z^\prime}^{(1)})>x\}\subseteq\{\tau_\gamma(x)<\infty,\Psi_{1:\tau_\gamma(x)}({Z^\prime}^{(M)})>x\}$. To see this, on $E_n$ we have that
\begin{align*}
\Psi_{1:n+1}({Z^\prime}^{(1)})&\geq\sum_{k=0}^{n-1}\underline{B}_{k+1}\prod_{j=1}^k A_j+\left(\underline{B}_{n+1}+{Z^\prime}^{(1)}A_{n+1}\right)\prod_{j=1}^n A_j\\
&\geq-\sum_{k=0}^{n-1}|\underline{B}_{k+1}|\prod_{j=1}^k A_j+\left(\underline{B}_{n+1}+xe^{n(\mu-\nu)-K}+{Z^\prime}^{(1)}A_{n+1}\right)\prod_{j=1}^n A_j-xe^{n(\mu-\nu)-K}\prod_{j=1}^n A_j\\
&\geq-\frac{e^{2K}}{1-e^{-\mu+2\nu}}+\min(\nu,1)\max\left(A_{n+1},\underline{B}_{n+1}+xe^{n(\mu-\nu)-K}\right)e^{-n(\mu+\nu)-K}-x\\
&\geq-\frac{e^{2K}}{1-e^{-\mu+2\nu}}+\min(\nu,1)\max\left(A_{n+1},\underline{B}_{n+1}\right)e^{-n(\mu+\nu)-K}-x\\
&\geq-\frac{e^{2K}}{1-e^{-\mu+2\nu}}+\min(\nu,1) x e^L-x>x,
\end{align*}
for sufficiently large $L$ that does not depend on $x$. Since $\{S_j(\gamma)\}_{j\leq n}$ is bounded by $K$, $\mu>\mu_\gamma$ and
\begin{align*}
S_{n+1}(\gamma)&=S_{n}(\gamma)+\log\left(\max\left(\bar{B}_{n+1}e^{-\gamma_{2}},A_{n+1}\right)\right)+\gamma_{1}\\
&>-n\left(\mu_\gamma+\nu\right)-K+\log\left(\max\left(\underline{B}_{n+1},A_{n+1}\right)\right)-\gamma_{2}+\gamma_{1}\\
&>\log x+n\left(\mu-\mu_\gamma\right)+L-\gamma_{2}+\gamma_{1}\\
&>\log x+L-\gamma_{2}+\gamma_{1}>s(x),
\end{align*}
for sufficiently large $L$ that does not depend on $x$, we can also conclude that $\tau_\gamma(x)=n+1<\infty$ by taking $x$ sufficiently large. This implies that
\begin{align}
\nonumber&\mathbb{P}\left(\tau_\gamma(x)<\infty,\Psi_{1:\tau_\gamma(x)}({Z^\prime}^{(M)})>x\right)
\nonumber\geq\sum_{n\geq0}\mathbb{P}(E_n)
\\
\nonumber
&\geq(1-\epsilon)\mathbb{P}({Z^\prime}^{(1)}>\nu)\sum_{n\geq0}\bigg\{\mathbb{P}\left(\max\left(A_1,\underline{B}_1\right)>xe^{n(\mu+\nu)+L+K}\right)
\\
&\hspace{130pt}
-\mathbb{P}\left(A_1>xe^{n(\mu+\nu)+L+K},\underline{B}_1<-xe^{n(\mu-\nu)+K}\right)\bigg\}.\label{SectionGCeq4}
\end{align}
From Assumption \ref{sec3ass2}c) we conclude that, for any $\epsilon^\prime>0$, by taking sufficiently large $x$, the following holds
\begin{align*}
\mathbb{P}\left(A_1>xe^{n(\mu+\nu)+L+K},\underline{B}_1<-xe^{n(\mu-\nu)+K}\right)&\leq \mathbb{P}\left(A_1>xe^{n(\mu-\nu)+K},\underline{B}_1<-xe^{n(\mu-\nu)+K}\right) \\
&\leq\epsilon^\prime\mathbb{P}\left(\max\left(A_1,\underline{B}_1\right)>xe^{n(\mu-\nu)+K}\right).
\end{align*}
Combining this with \eqref{SectionGCeq4}, we obtain that
\begin{align}
\nonumber&\mathbb{P}\left(\tau_\gamma(x)<\infty,\Psi_{1:\tau_\gamma(x)}({Z^\prime}^{(M)})>x\right)\\
\nonumber
&\geq(1-\epsilon)\mathbb{P}({Z^\prime}^{(1)}>\nu)\sum_{n\geq0}\bigg\{\mathbb{P}\left(\max\left(A_1,\underline{B}_1\right)>xe^{n(\mu+\nu)+L+K}\right)
\\
&\hspace{130pt}
-\epsilon^\prime\mathbb{P}\left(\max\left(A_1,\underline{B}_1\right)>xe^{n(\mu-\nu)+K}\right)\bigg\}.\label{SectionGCeq5}
\end{align}
Since $\mathbb{P}\left(\max\left(A_1,\underline{B}_1\right)>y\right)$ is decreasing in $y$, we observe that, for any $n$,
$$\mathbb{P}\left(\max\left(A_1,\underline{B}_1\right)>xe^{n(\mu+\nu)+L+K}\right)\geq\frac{1}{\mu+\nu}\int_{\log x+L+K+n(\mu+\nu)}^{\log x+L+K+(n+1)(\mu+\nu)}\mathbb{P}\left(\log\max\left(A_1,\underline{B}_1\right)>y\right)dy,$$
and that
$$\mathbb{P}\left(\max\left(A_1,\underline{B}_1\right)>xe^{n(\mu-\nu)+K}\right)\leq\frac{1}{\mu-\nu}\int_{\log x+K+(n-1)(\mu-\nu)}^{\log x+K+n(\mu-\nu)}\mathbb{P}\left(\log\max\left(A_1,\underline{B}_1\right)>y\right)dy.$$
Moreover, using the fact that $\bar{F}_I$ is long tailed, we obtain from \eqref{SectionGCeq5} that
\begin{align}
\nonumber&\mathbb{P}\left(\tau_\gamma(x)<\infty,\Psi_{1:\tau_\gamma(x)}({Z^\prime}^{(M)})>x\right)\\
\nonumber
&\geq(1-\epsilon)\mathbb{P}({Z^\prime}^{(1)}>\nu)\left(\frac{1}{\mu+\nu}\bar{F}_I(\log x+L+K)-\frac{\epsilon^\prime}{\mu-\nu}\bar{F}_I(\log x+L+K-(\mu-\nu))\right)\\
\nonumber
&\sim(1-\epsilon)\mathbb{P}({Z^\prime}^{(1)}>\nu)\left(\frac{1}{\mu+\nu}-\frac{\epsilon^\prime}{\mu-\nu}\right)\bar{F}_I(\log x)
\\
&\sim\mu(1-\epsilon)\mathbb{P}({Z^\prime}^{(1)}>\nu)\left(\frac{1}{\mu+\nu}-\frac{\epsilon^\prime}{\mu-\nu}\right)\mathbb{P}(T(x)<\infty)\label{SectionGCeq2}.
\end{align}
where in \eqref{SectionGCeq2} we use \citeA[Theorem 3.1]{dyszewski2016}. Letting $\epsilon,\epsilon^\prime,\nu\to0$ we obtain the result. This result implies that the relative bias converge to $0$, since the numerator in \eqref{SectionGCeq1} is always smaller than the denominator.
\end{proof}

Let the conditions in Theorem \ref{sec4thm1} and Theorem \ref{sec5thm1} be satisfied. The following algorithm for estimating $\mathbb{P}(Z>x)$ has bounded relative error and vanishing relative bias as $x\to \infty$.  
\begin{algm}
\begin{description}
\item[]
\begin{description}
\item[]
\item[STEP 0.] For fixed $\delta\in(0,1)$, set $a_*\longleftarrow a_*(\delta)\leq0$ satisfying \eqref{sec2eq3}.
\item[STEP 1.] Initialize $s\longleftarrow0$, $z\longleftarrow 1$ and $L\longleftarrow1$.
\item[STEP 2.] Set $s^\prime\longleftarrow s$ and $z^\prime\longleftarrow z$. Run Algorithm \ref{sec2algm1} until the random walk $S_n(\gamma)$ crosses $s(x)$. Meanwhile, update $s$ and $L$ according to STEP 2 of Algorithm \ref{sec2algm1}, then update $z$ via the backward iteration.
\item[STEP 3.] Set $i\longleftarrow0$, $s^\prime\longleftarrow s$ and $z^\prime\longleftarrow z$. While $i<M$, update $z$ via the backward iteration, $z^\prime\longleftarrow z$ and $i\longleftarrow i+1$. 
\item[STEP 4.] If $z>x$ then return $L$. Otherwise, return $0$.
\end{description}
\end{description}
\end{algm}

\section {Truncation Index and Unbiased Estimator}\label{SectionTIUE}
In this section we analyze the asymptotic behavior of the relative bias as $M\to\infty$ for fixed $x$,
based on which we propose an unbiased estimator for $\mathbb{P}(Z>x)$ using the technique studied in \citeA{rheeglynn2015}. We first go back to the Example \ref{sec3ex1}---i.e., the stochastic perpetuity example with $B_n=1$.

\addtocounter{example}{-1}
\begin{example}[continued]
Consider the Markov chain given by \eqref{sec3eq0}. From \eqref{sec5eq1} we know that in order to analyze the relative bias we need to consider the term $\Theta(x,M)$, which is given by
$$\Theta(x,M)=\mathbb{P}\left(A^\prime_x{Z^\prime}^{(M)}+B^\prime_x\leq x,A^\prime_xZ^\prime+B^\prime_x>x\, \middle|\, \tau_\gamma(x)<\infty\right).$$
Let $\mathbb{P}^{(x)}\left(\cdot\right)$ denote the conditional probability of $\mathbb{P}\left(\cdot\,\middle|\,\tau_\gamma(x)<\infty\right)$ and $\mathbb E^{(x)}$ the corresponding expectation operator.
Note that
\begin{align}
\nonumber\Theta(x,M)
&=\int \mathbbm{1}_{\left\{{Z^\prime}^{(M)}\leq \frac{x-B^\prime_x}{A^\prime_x},Z^\prime>\frac{x-B^\prime_x}{A^\prime_x}\right\}}d\mathbb{P}^{(x)}\\
\nonumber&=\int \mathbb{P}^{(x)}\left({Z^\prime}^{(M)}\leq y,Z^\prime>y\right) \mathbb{P}^{(x)}_{\xi_x}\left(dy\right)\\
&=\int \left\{\mathbb{P}^{(x)}\left(Z^\prime>y\right)-\mathbb{P}^{(x)}\left({Z^\prime}^{(M)}>y\right)\right\} \mathbb{P}^{(x)}_{\xi_x}\left(dy\right),
\label{sec6eq1}
\end{align}
where \eqref{sec6eq1} follows from the fact that $\left\{{Z^\prime}^{(M)}>y\right\}\subseteq\left\{Z^\prime>y\right\}$. Using the strong Markov property we have that ${Z^\prime}^{(M)}\,{\buildrel d \over =}\,Z^{(M)}$ and $Z^\prime\,{\buildrel d \over =}\,Z$ under $\mathbb{P}^{(x)}$. Therefore, we can write \eqref{sec6eq1} as
\begin{align*}
\Theta(x,M)&=\int \left\{\mathbb{P}\left(Z>y\right)-\mathbb{P}\left(Z^{(M)}>y\right)\right\} \mathbb{P}^{(x)}_{\xi_x}\left(dy\right).
\end{align*}
Combining this with the fact that the backward iteration $Z^{(M)}$ has the same distribution as $Z_M$, we obtain that
\begin{equation}\label{sec6eq2}
\Theta(x,M)=\int \left\{\mathbb{P}\left(Z>y\right)-\mathbb{P}\left(Z_M>y\right)\right\} \mathbb{P}^{(x)}_{\xi_x}\left(dy\right) \leq d_{TV}(Z_M, Z)
\end{equation}
where $d_{TV}$ denotes the total variation distance. 
To get a handle on this quantity, we apply the Lyapunov criterion in \citeA[Theorem 3.6]{jarnerroberts2002}, which implies a polynomial convergence rate of the $M$-step transition kernel to the invariant distribution in the total variation norm. 
We assume that the Markov chain $\{Z_n\}_{n\in\mathbb{N}}$ given by \eqref{sec3eq0} is irreducible and aperiodic; this is the case, for example, if $A_1$ has a Lebesgue density (\citeNP[Lemma 2.2.2]{buraczewskidamekmikosch2016}). Moreover, assume that there exists an integer $q\geq2$ such that $\mathbb{E}|X_1|^q<\infty$.
In order to establish the Lyapunov condition, let $V(x)=1\vee(\log x)^q$. 
Note that $V(x)=(\log x)^q \mathbbm 1_{\{x> e\}} + \mathbbm 1_{\{x\leq e\}}$ and hence the binomial expansion gives
\begin{align*}
PV(x)
&=
\mathbb{E}\big[\left(\log\left(A_1x+1\right)\right)^q\I{A_1x+1>e}+ \I{A_1x +1\leq e}\big]
\\
&=\E\left[ \left(\log\frac{A_1 x +1}{x} + \log x\right)^q\I{A_1x + 1> e} + \I{A_1x +1\leq e}\right]
\\
&=
\E\left[(\log x)^q \I{A_1x+1>e} + \sum_{i=1}^q {q \choose i}(\log x)^{q-i} \left(\log\frac{A_1x + 1}{x}\right)^{i} \I{A_1x +1 > e}+ \I{A_1x+1 \leq e}\right]
\\
&=
V(x) + \E(\log x)^q (\I{A_x x + 1>e} - \I{x>e}) + \E (\I{A_1 x+1 \leq e}-\I{x\leq e})
\\
&\hspace{15pt} 
+ q(\log x)^{q-1} \E \left(\log \frac{ A_1 x+1}{x}\I{A_1 x + 1 > e}\right)
+\E\sum_{i=2}^q {q \choose i}(\log x)^{q-i} \left(\log\frac{A_1x + 1}{x}\right)^{i} \I{A_1x +1 > e}
\end{align*}
For $x > e$,
\begin{align*}
PV(x) 
&\leq V(x) + \P(A_1 x + 1 \leq e)
+ q(\log x)^{q-1} \E \left(\log \frac{ A_1 x+1}{x}\I{A_1 x + 1 > e}\right)
\\
&\hspace{180pt}
+\sum_{i=2}^q {q \choose i}(\log x)^{q-i} \,\E\left[\left(\log\frac{A_1x + 1}{x}\right)^{i} \I{A_1x +1 > e}\right].
\end{align*}
Note that $\left|\log \frac{ A_1 x+1}{x}\I{A_1 x + 1 > e}\right| \leq |\log A_1|+C$ for some constant $C$ and noting that the right-hand-side doesn't depend on $x$ and has finite $q$-th moment, 
there has to be $c_i$'s such that 
$$
\sum_{i=2}^q {q \choose i}(\log x)^{q-i} \,\E\left[\left(\log\frac{A_1x + 1}{x}\right)^{i} \I{A_1x +1 > e}\right]
\leq \sum_{i=0}^{q-2} c_i (\log x)^{i} \leq \epsilon (\log x)^{q-1}
$$
for sufficiently large $x$.
On the other hand, note that $\log \frac{ A_1 x+1}{x}\I{A_1 x + 1 > e}$ converges to $X_1 = \log A_1$ almost surely as $x\to \infty$ and hence by dominated convergence $\E \left(\log \frac{ A_1 x+1}{x}\I{A_1 x + 1 > e}\right)\to \E \log A_1 < 0$. Therefore, for any fixed $\epsilon>0$, 
$$
q(\log x)^{q-1} \E \left(\log \frac{ A_1 x+1}{x}\I{A_1 x + 1 > e}\right)  \leq (q\E X_1 + \epsilon) (\log x)^{q-1}
$$
for sufficiently large $x$. Choosing $\epsilon$ so that $q\E X_1 + 3\epsilon <0$ and noting that $\P(A_1x+1\leq e)\to 0$ as $x\to\infty$, as well as $(\log x)^{q-1} =\big( (\log x)^q \I{x>e} + \I{x\leq e}\big)^{\frac{q-1}{q}}$ for $x>e$, we conclude that there exists $K$ such that
\begin{align*}
PV(x) 
&\leq 
V(x) + \epsilon(\log x)^{q-1} + (q\E X_1 + \epsilon)(\log x)^{q-1} + \epsilon(\log x)^{q-1}
\\
&\leq
V(x) - cV^{({q-1})/{q}}(x)
\end{align*}
for $x>K$, where $c = -(q\E X_1 + 3\epsilon)>0$.
Finally, since $PV(x)$, $V(x)$ and $V^{(q-1)/q}(x)$ are bounded on $C=[0,K]$, there exists a constant $b$ such that
$$PV(x)\leq V(x)-cV^{(q-1)/q}(x)+b\mathbbm{1}_{C},$$
which is the sufficient condition in \citeA[Theorem 3.6]{jarnerroberts2002} for polynomial ergodicity; we conclude that the $M$-step transition kernel converges to the stationary distribution in the total variation norm at a polynomial rate with order $q-1$, i.e., there exists a constant $\kappa^\prime$ satisfying
\begin{equation}\label{polynomial_convergence}
\Theta(x,M)\leq d_{TV}(Z_M,Z)<\kappa M^{-(q-1)},
\end{equation}
for all $M\in\mathbb{N}$. 
It should be noted that an exact expression of the constant $\kappa$ can be obtained in a few special cases---for example, see e.g.\ \citeA{doucmoulinessoulier2007}, \citeA{kalashnikovtsitsiashvili1999} and the references therein. However, applying the method studied in \citeA{rheeglynn2015} we can get rid of this constant altogether and obtain an unbiased, strongly efficient estimator. 
In order to apply the method, a sufficient condition is to bound
\begin{equation}\label{sec6eq5}
\frac{\mathbb{E}^{Q^\gamma_{a_*}}(L_T^\Delta(x,M)-L_T(x))^2}{\mathbb{P}(Z>x)^2}
\end{equation}
by a decreasing function of $M$ independent of $x$. Once we can have such a bound, we can construct an unbiased estimator that is given by
\begin{equation}\label{sec6eq8}
L^{\text{RG}}_T(x)\triangleq\sum_{i=0}^N \frac{L_T^\Delta(x,2^i) - L_T^\Delta(x,2^{i-1})}{\mathbb{P}(N\geq i)},
\end{equation}
whose second moment is
$$\sum_{i=0}^\infty \frac{\mathbb{E}^{Q^\gamma_{a_*}}(L_T^\Delta(x,2^{i-1})-L_T(x))^2-\mathbb{E}^{Q^\gamma_{a_*}}(L_T^\Delta(x,2^i)-L_T(x))^2}{\mathbb{P}(N\geq i)},$$
where $N$ is a random truncation index independent of everything else and $L_T^\Delta(x,2^i)$ is interpreted as $0$ if $i<0$ (for details see \citeNP[Theorem 1]{rheeglynn2015}). It turns out that such a bound on \eqref{sec6eq5} can be derived easily, if
\begin{equation}\label{sec6eq6}
\sup_{x>0} \frac{\mathbb{E}^{Q^\gamma_{a_*}}L_T^{2+\epsilon}(x)}{\mathbb{P}(Z>x)^{2+\epsilon}} < \infty
\end{equation}
holds for some $\epsilon>0$. In case this is possible for some positive $\epsilon>0$, we can proceed as follows: let $\mathcal{E}_x(i)= A'_xZ'^{(2^i)} + B'_x$ and $\mathcal{E}_x = A'_xZ'+B'_x$. For $\beta\in(0,1)$, using the H\"{o}lder's inequality we get that
\begin{align}
\nonumber\frac{\mathbb{E}^{Q^\gamma_{a_*}}\left[(L_T^\Delta(x,2^i)-L_T(x))^2\right]}{P(Z>x)^2}=&\frac{\mathbb{E}^{Q^\gamma_{a_*}}\left[\mathbbm{1}_{\{\tau_\gamma(x)<\infty, \mathcal{E}_x(i) \leq x, \mathcal{E}_x > x\}}  (M_{\tau_\gamma}^{-1}(x))^{2}\right]}{\mathbb{P}(Z>x)^2}\\
\nonumber=&\frac{\mathbb{E}^{Q^\gamma_{a_*}} \left[\left(\mathbbm{1}_{\{\tau_\gamma(x)<\infty, \mathcal{E}_x(i) \leq x, \mathcal{E}_x > x\}}M_{\tau_\gamma}^{-1}(x)\right)^{\beta}\left(\mathbbm{1}_{\{T(x)<\infty\}} M_{\tau_\gamma}^{-1}(x)\right)^{2-\beta}\right]}{\mathbb{P}(Z>x)^2}\\
\nonumber\leq&\frac{\mathbb{E}^{Q^\gamma_{a_*}}\left[\mathbbm{1}_{\{\tau_\gamma(x)<\infty, \mathcal{E}_x(i) \leq x, \mathcal{E}_x > x\}}  M_{\tau_\gamma}^{-1}(x)\right]^\beta}{\mathbb{P}(Z>x)^\beta}\frac{\mathbb{E}^{Q^\gamma_{a_*}}\left[ \mathbbm{1}_{\{T(x)<\infty\}} M_{\tau_\gamma}^{-1}(x)^{\frac{2-\beta}{1-\beta}}\right]^{1-\beta}}{\mathbb{P}(Z>x)^{2-\beta}}\\
=&\left[\frac{\mathbb{E}^{Q^\gamma_{a_*}} \mathbbm{1}_{\{\tau_\gamma(x)<\infty, \mathcal{E}_x(i) \leq x, \mathcal{E}_x > x\}}  M_{\tau_\gamma}^{-1}(x)}{\mathbb{P}(Z>x)}\right]^\beta\left[\frac{\mathbb{E}^{Q^\gamma_{a_*}} L_T^{\frac{2-\beta}{1-\beta}}(x)}{\mathbb{P}(Z>x)^\frac{{2-\beta}}{1-\beta}}\right]^{1-\beta}.\label{sec6eq9}
\end{align}
The first term in \eqref{sec6eq9} is bounded by $\left(\kappa2^{-i(q-1)}\right)^\beta$ due to \eqref{polynomial_convergence}, and the second term is uniformly bounded w.r.t.\ $x$. 
Therefore, the second moment relative to $P(Z>x)^2$ can be bounded uniformly w.r.t.\ $x$, if we choose $\beta$ and $N$ in a suitable way. 
For example, setting $\beta \triangleq \frac{1+2\epsilon}{q-1}$ where $\epsilon>0$ is small enough to ensure that $\beta<1$, and then $N$ such that $\mathbb{P}(N\geq i)=2^{-i(1+\epsilon)}$, one can achieve the purpose. Finally, we are left with verifying \eqref{sec6eq6}, which turns out to be possible by adapting the idea as described in \citeA{blanchetglynn2008}. Using the same argument as in the proof of Theorem \ref{sec4thm1}, it is sufficient to analyze the term that is given by
\begin{equation}\label{sec6eq7}
\frac{\mathbb{E}^{Q^\gamma_{a_*}}L_{\tau_\gamma}^{2+\epsilon}(x)}{\mathbb{P}\left(\tau_\gamma(x)<\infty\right)^{2+\epsilon}}.
\end{equation}
The following Lemma, which can be considered as an extension of Theorem \ref{sec2thm1}, proves that the estimator $L_T^{RG}(x)$ defined in \eqref{sec6eq8} is an unbiased estimator of $\P(Z>x)$ and is strongly efficient. 

\begin{lemma}\label{sec6lem1}
Suppose that $\mathbb{E}X_1<0$ and $X_1$ belongs to $S^*$. Fix $\gamma\in(0,-\mathbb{E}X_1)$. Let $v_\gamma$ and $w_\gamma$ be defined as in \eqref{sec4eq4} and \eqref{sec4eq5}. Let $\epsilon>0$. 
For fixed $\delta\in(0,1)$, one can choose $a_*=a_*(\delta)\leq0$ so that 
$$-\delta\leq\frac{v^{2+\epsilon}_\gamma(y)-w_\gamma^{2+\epsilon}(y)}{\mathbb{P}(X_1>-y)w^{1+\epsilon}_\gamma(y)},\quad \forall y\leq s(x)+a_*.$$
Let 
$$L_{\tau_\gamma}(x)\triangleq \mathbbm{1}_{\{\tau_\gamma(x)<\infty\}}\prod_{k=1}^{\tau_\gamma(x)}\frac{w_\gamma(S_{k-1}+a_*)}{v_\gamma(S_k+a_*)}.$$
Then $L_{\tau_\gamma}(x)$ is an unbiased estimator of $\mathbbm{P}(\max_{n\in\mathbbm{N}}S_n>s(x))$ under $\P^{Q_{{a_*}}}$, and 
$$\sup_{x>0}\frac{\mathbb{E}^{Q_{a_*}}L_{\tau_\gamma}^{2+\epsilon}(x)}{\mathbb{P}\left(\max\limits_{n\geq0}S_n>s(x)\right)^{2+\epsilon}}<\infty,$$
where $\P^{Q_{{a_*}}}$ denotes the probability measure associated with the random walk $\{S_n\}_{n\in\mathbb{N}}$ having the one-step transition kernel
$$Q_{a_*}(y,dz)=P(y,dz)\frac{v_\gamma(z+a_*)}{w_\gamma(y+a_*)},$$
and $\mathbb{E}^{Q_{a_*}}$ denotes the corresponding expectation operator.
\end{lemma}

\begin{proof}
The proof including the existence of $a_*$ can be found in Appendix B.
\end{proof}

\end{example}

Note that the above discussion can easily be extended to cover the general stochastic recursion. We conclude this section with the following theorem. 

\begin{thm}\label{sec6thm1}
Suppose that $\mathbb{E}S_1(\gamma)<0$, $S_1(\gamma)$ belongs to $S^*$ and Assumption 2 holds. Let $v_\gamma$ and $w_\gamma$ be defined as in \eqref{sec4eq4} and \eqref{sec4eq5}. For fixed $\delta\in(0,1)$ and $\beta\in(0,1)$, set $a_*=a_*(\delta)\leq0$ satisfying
$$-\delta\leq\frac{v^{\frac{2-\beta}{1-\beta}}_\gamma(y)-w_\gamma^{\frac{2-\beta}{1-\beta}}(y)}{\mathbb{P}(X_1>-y)w^{\frac{1}{1-\beta}}_\gamma(y)},\quad \forall y\leq s(x)+a_*.$$
Moreover, assume that $\E |\log A_1|^q + \E|\log \bar B_1|^q< \infty$. Then, it is possible to choose $N$ independently of $x$, such that
$$\sum_{i=0}^\infty \frac{\mathbb{E}^{Q^\gamma_{a_*}}(L_T^\Delta(x,2^{i-1})-L_T(x))^2}{\mathbb{P}(Z>x)^2\mathbb{P}(N\geq i)}$$
converges, and hence, the estimator $L^{RG}_T(x)$ defined in \eqref{sec6eq8} is unbiased and strongly efficient.
\end{thm}

\begin{rmk}
Note that, in general, the bias elimination scheme of \citeA{rheeglynn2015} is not guaranteed to produce non-negative estimators, which might not be ideal in the context of estimating (rare event) probabilities. However, in our case, $L_T^\Delta(x,M)$ increases w.r.t.\ $M$, and hence, the resulting unbiased estimator $L^{RG}_T(x)$ is always non-negative. 
\end{rmk}

\section{Numerical Results}
Here we investigate our algorithm numerically based on a stochastic perpetuity with $B_n=1$. We consider the increment $X_n \,{\buildrel d \over =}\, \mathcal W -3/2$ where $\mathcal{W}$ is a random variable with Weibull distribution:
$$\mathbbm{P}(\mathcal{W}>t)=\exp\left(-2t^{1/2}\right).$$
For the algorithmic parameters, we chose $a_*=-10$, $\gamma=0.5$. Moreover, we use a geometric distributed random truncation index with parameter $0.5$. Figure \ref{sec4fig1} shows the change of estimated probability with respect to the different choice of $M$ for 4 different values of $x=10^8$, $x=10^{16}$, $x=10^{32}$, and $x=10^{64}$ in each of the four plots. One can see that the estimated probability stabilizes as $M$ grows, which suggests that our estimator is consistent as $M\to \infty$. Comparing the four plots, one can also tell that the initial bias for small $M$ decrease as $x$ increases, which is consistent with the conclusion of Theorem \ref{sec5thm1} (vanishing relative bias). Table 1 reports the estimated probabilities, their $95\%$-confidence intervals and the estimated coefficients of variation, that is, the estimated standard deviation divided by the sample mean (based on $200000$ samples), for different values of $x$ and $M$. 
In the last column, we present the results produced with the unbiased algorithm as introduced in Section \ref{SectionTIUE}. 
We can see that, on the one hand the ratio between the estimated probability and the standard deviation stays roughly constant over a range of $x$ values and $M$ values; on the other hand, the estimated probability using the fix truncation method tend to converge to the estimated probability produced with the unbiased algorithm as $M$ grows. 
These observations are consistent with the strong efficiency---predicted in Theorem \ref{sec4thm1} and Theorem \ref{sec6thm1}---of our estimators.

\begin{figure}[]
\centering
\includegraphics[scale=0.5]{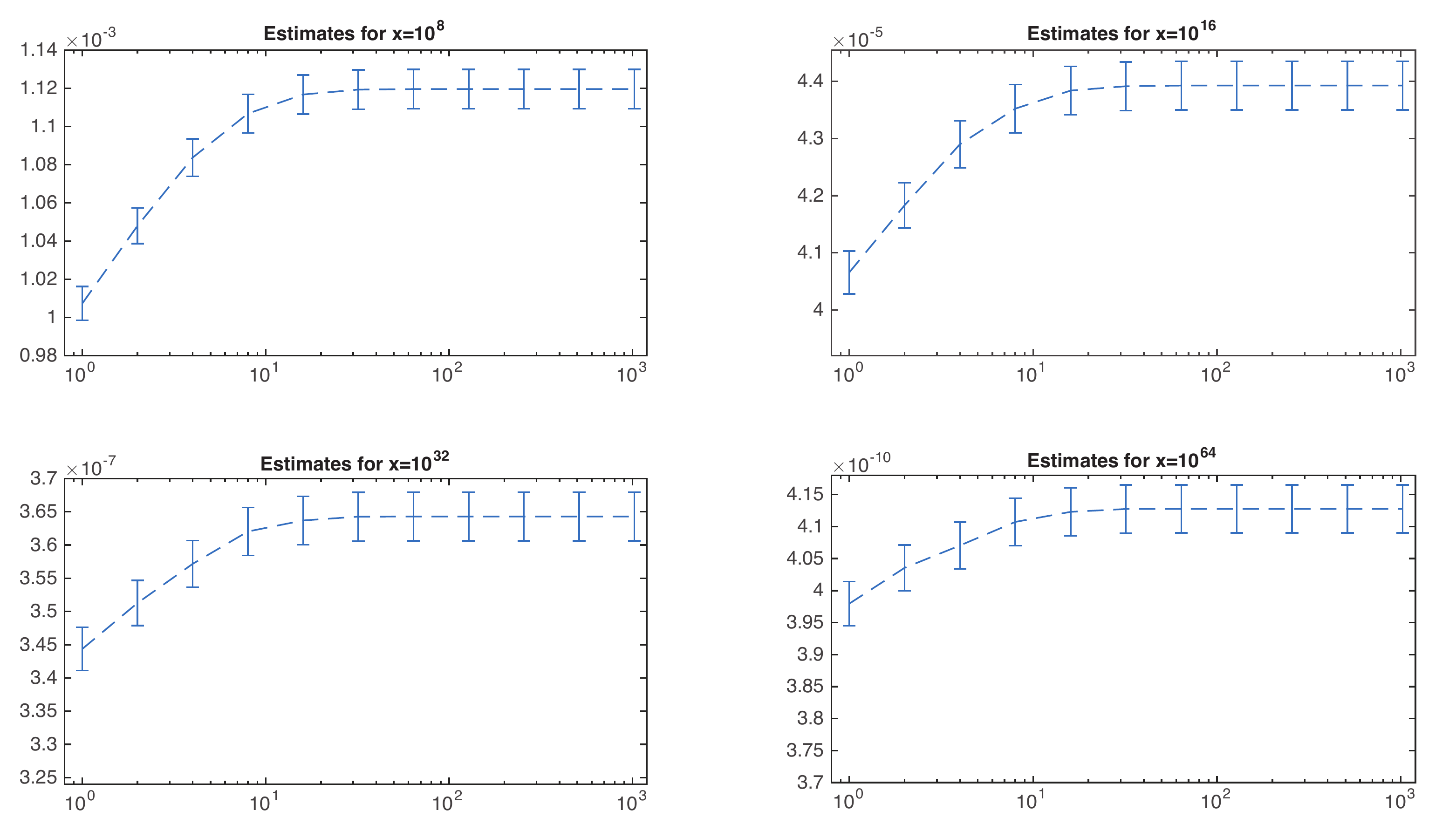}
\caption{Estimated probabilities for changing values of $M$. The $y$-axis values indicate the estimated rare-event probabilities and the vertical bars indicate the $95\%$ confidence intervals. The $x$-axis values indicate the truncation index $M$. In each of the subplot, we can see that as $M$ increases, the estimated probability converges to a fixed value, which suggests that our estimator is consistent with respect to $M$. Comparing the four subplots for different value of $x$, one can see that the relative bias for small $M$ decreases as $x$ grows.}
\label{sec4fig1}
\end{figure}

\begin{table}\label{sec4tab1}
\resizebox{\textwidth}{!}{
\centering
\begin{tabular}{llllll}
\hline
\begin{tabular}[c]{@{}l@{}l@{}}Est\\ CI\\ CV\end{tabular} & $\hspace{8.5pt}M=2^{2}$ & $\hspace{8.5pt}M=2^{4}$ & $\hspace{8.5pt}M=2^{6}$ & $\hspace{8.5pt}M=2^{8}$ & \hspace{8.3pt} RG
\\ \hline
$x=10^{8}$&\begin{tabular}[c]{@{}l@{}l@{}}$\hspace{8.5pt}1.083\times10^{-3}$\\$\pm 0.009\times 10^{-3}$\\ \hspace{8.5pt}$2.06$\end{tabular}&\begin{tabular}[c]{@{}l@{}l@{}}\hspace{8.5pt}$1.117\times10^{-3}$\\$\pm 0.010\times 10^{-3}$\\ \hspace{8.5pt}$2.10$\end{tabular}&\begin{tabular}[c]{@{}l@{}l@{}}\hspace{8.5pt}$1.120\times10^{-3}$\\$\pm 0.010\times 10^{-3}$\\ \hspace{8.5pt}$2.10$\end{tabular}&\begin{tabular}[c]{@{}l@{}l@{}}\hspace{8.5pt}$1.120\times10^{-3}$\\$\pm 0.010\times 10^{-3}$\\ \hspace{8.5pt}$2.10$\end{tabular}&\begin{tabular}[c]{@{}l@{}l@{}}\hspace{8.5pt}$1.119\times10^{-3}$\\$\pm 0.013\times 10^{-3}$\\ \hspace{8.5pt}$2.70$\end{tabular}
\\ \hline
$x=10^{16}$&\begin{tabular}[c]{@{}l@{}l@{}}\hspace{8.5pt}$4.271\times10^{-5}$\\$\pm 0.041\times 10^{-5}$\\ \hspace{8.5pt}$2.17$\end{tabular}&\begin{tabular}[c]{@{}l@{}l@{}}\hspace{8.5pt}$4.373\times10^{-5}$\\$\pm 0.042\times 10^{-5}$\\ \hspace{8.5pt}$2.22$\end{tabular}&\begin{tabular}[c]{@{}l@{}l@{}}\hspace{8.5pt}$4.383\times10^{-5}$\\$\pm 0.043\times 10^{-5}$\\ \hspace{8.5pt}$2.22$\end{tabular}&\begin{tabular}[c]{@{}l@{}l@{}}\hspace{8.5pt}$4.383\times10^{-5}$\\$\pm 0.043\times 10^{-5}$\\ \hspace{8.5pt}$2.22$\\ \end{tabular}&\begin{tabular}[c]{@{}l@{}l@{}}\hspace{8.5pt}$4.375\times10^{-5}$\\$\pm 0.053\times 10^{-5}$\\ \hspace{8.5pt}$2.76$\end{tabular}
\\ \hline
$x=10^{32}$&\begin{tabular}[c]{@{}l@{}l@{}}\hspace{8.5pt}$3.583\times10^{-7}$\\$\pm 0.035\times 10^{-7}$\\ \hspace{8.5pt}$2.25$\end{tabular}&\begin{tabular}[c]{@{}l@{}l@{}}\hspace{8.5pt}$3.646\times10^{-7}$\\$\pm 0.037\times 10^{-7}$\\ \hspace{8.5pt}$2.28$\end{tabular}&\begin{tabular}[c]{@{}l@{}l@{}}\hspace{8.5pt}$3.650\times10^{-7}$\\$\pm 0.037\times 10^{-7}$\\ \hspace{8.5pt}$2.29$\end{tabular}&\begin{tabular}[c]{@{}l@{}l@{}}\hspace{8.5pt}$3.650\times10^{-7}$\\$\pm 0.037\times 10^{-7}$\\ \hspace{8.5pt}$2.29$\end{tabular}&\begin{tabular}[c]{@{}l@{}l@{}}\hspace{8.5pt}$3.663\times10^{-7}$\\$\pm 0.045\times 10^{-7}$\\ \hspace{8.5pt}$2.81$\end{tabular}
\\ \hline
$x=10^{64}$&\begin{tabular}[c]{@{}l@{}l@{}}\hspace{8.5pt}$4.079\times10^{-10}$\\$\pm 0.037\times 10^{-10}$\\ \hspace{8.5pt}$2.05$\end{tabular}&\begin{tabular}[c]{@{}l@{}l@{}}\hspace{8.5pt}$4.120\times10^{-10}$\\$\pm 0.037\times 10^{-10}$\\ \hspace{8.5pt}$2.06$\end{tabular}&\begin{tabular}[c]{@{}l@{}l@{}}\hspace{8.5pt}$4.123\times10^{-10}$\\$\pm 0.038\times 10^{-10}$\\ \hspace{8.5pt}$2.06$\end{tabular}&\begin{tabular}[c]{@{}l@{}l@{}}\hspace{8.5pt}$4.123\times10^{-10}$\\$\pm 0.038\times 10^{-10}$\\ \hspace{8.5pt}$2.06$\end{tabular}&\begin{tabular}[c]{@{}l@{}l@{}}\hspace{8.5pt}$4.115\times10^{-10}$\\$\pm 0.041\times 10^{-10}$\\ \hspace{8.5pt}$2.27$\end{tabular}
\\ \hline
\end{tabular}}
\caption{Estimated rare-event probability and $95\%$ confidence intervals. One can see that, on the one hand the ratio between the standard deviation and the estimated probability stays roughly constant for different combinations of $x$ and $M$, on the other hand, as $M$ grows, the estimated probability produced with the fix truncation method tends to converge to the estimated probability produced with the unbiased algorithm, which suggests the consistency and the strong efficiency of our estimators predicted by Theorem \ref{sec4thm1} and Theorem \ref{sec6thm1}.}
\end{table}

\appendix
\renewcommand{\appendixpagename}{Appendix}
\appendixpage

\section{Proof of Proposition \ref{sec5corol1}}
First we claim that, conditional on $\{\tau_\gamma(x)<\infty\}$, the first term in $\xi_x$ converges to $0$ in probability. Let $\epsilon>0$ be arbitrary, we have that
$$\mathbb{P}\left(\frac{\log{\left(1-\frac{B^\prime_x}{x}\right)}}{a(\log x)}\leq-\epsilon\,\middle|\, \tau_\gamma(x)<\infty\right)=\mathbb{P}\left(\frac{B^\prime_x}{x}\geq 1-x^{-\frac{\epsilon}{\alpha}}\,\middle|\, \tau_\gamma(x)<\infty\right).$$
Let $M_{\tau_\gamma(x)}$ denote the maximum of $\{S_i(\gamma)\}_{i\leq\tau_\gamma(x)-1}$. It is well known that $M_{\tau_\gamma(x)}=\mathcal{O}(1)$  (cf.\ the proof of Theorem 1.1 in \citeNP{asmussenklueppelberg1996}). Moreover, $B^\prime_x$ is bounded by
\begin{align*}
B^\prime_x&=\sum_{n=0}^{\tau_\gamma(x)-1}e^{S_n}=\sum_{n=0}^{\tau_\gamma(x)-1}e^{S_n(\gamma)-n\gamma}\leq e^{M_{\tau_\gamma(x)}}\sum_{n=0}^{\tau_\gamma(x)-1}e^{-n\gamma}\leq e^{M_{\tau_\gamma(x)}}(1-e^{-\gamma})^{-1}.
\end{align*}
Therefore, for $\alpha>0$, we have that
\begin{align}
\nonumber\mathbb{P}\left(\frac{B^\prime_x}{x}\geq1-x^{-\frac{\epsilon}{\alpha}}\,\middle|\, \tau_\gamma(x)<\infty\right)&\leq\mathbb{P}\left(M_{\tau_\gamma(x)}-\log x-\log(1-e^{-\gamma})\geq\log(1-x^{-\frac{\epsilon}{\alpha}})\,\middle|\, \tau_\gamma(x)<\infty\right)\\
\nonumber&=\mathbb{P}\left(\frac{M_{\tau_\gamma(x)}}{a(\log x)}-\frac{\log x}{a(\log x)}-\frac{\log(1-e^{-\gamma})}{a(\log x)}\geq\frac{\log(1-x^{-\frac{\epsilon}{\alpha}})}{a(\log x)}\,\middle|\, \tau_\gamma(x)<\infty\right)\\
&\to0\label{sec4eq9},
\end{align}
as $x\to\infty$, since $\log x/a(\log x)=\alpha$. The convergence of the last two terms in $\xi_x$ is an application of \citeA[Theorem 1.1]{asmussenklueppelberg1996}.

\section{Proof of Lemma \ref{sec6lem1}}
First notice that the numerator in \eqref{sec6eq7} is equal to $\mathbb{E}L_{\tau_\gamma}^{1+\epsilon}(x)$. Analogous to \citeA[Theorem 2]{blanchetglynn2008}, we can derive a similar result for $\mathbb{E}L_{\tau_\gamma}^{1+\epsilon}(x)$ simply by replacing $r$ with $r^{1+\epsilon}$, where
$$r(y,z)=\frac{v(z)}{w(y)},$$
and $v$, $w$ are defined in the corresponding way as in \eqref{sec2eq7}. Based on this observation we claim the following Proposition (cf.\ \citeNP[Proposition 2]{blanchetglynn2008}):

\begin{prop}\label{appendixprop1}
Let $P(y,dz)$ denote the transition kernel of the random walk $\{S_n\}_{n\in\mathbb{N}}$. Assume that $w(y)>0$ for $y\leq s(x)$ and that there exists a finite-valued function $h\,:\mathbb{R}\longrightarrow[\delta_1,\infty)$ satisfying
\begin{equation}\label{appendixeq1}
w^{1+\epsilon}(y)\int v(z)h(z)P(y,dz)\leq h(y)v^{2+\epsilon}(y),
\end{equation}
for $y\leq s(x)$. If $h(z)\geq1$ for $z>s(x)$ and $v(z)\geq\delta_2>0$ for $z>s(x)$, then we have that
$$\mathbb{E}L_{\tau_\gamma}^{1+\epsilon}(x)\leq\delta_1^{-1}\delta_2^{-(2+\epsilon)}v^{2+\epsilon}(y)h(y).$$
\end{prop}

Recall that $v_\gamma$ defined as in \eqref{sec4eq4} depends implicitly on $x$. The Pakes-Veraverbeke's Theorem implies that $\mathbb{P}(\tau_\gamma(x)<\infty)\sim v_\gamma(y)$ for every fixed $y$ as $x\to\infty$. This observation gives us a way to prove Lemma \ref{sec6lem1}: first find a suitable function $h$ such that \eqref{appendixeq1} holds with $v$ and $w$ being defined as in \eqref{sec4eq1} and \eqref{sec4eq2}, then the result can be yielded using Proposition \ref{appendixprop1}. Define
$$h(y)=\mathbbm{1}_{(-\infty,s(x)-a_*]}(y)+(1-\delta)\mathbbm{1}_{(s(x)-a_*,\infty)}(y).$$
Now \eqref{appendixeq1} is equivalent to (cf.\ the proof of Theorem 3 in \citeNP{blanchetglynn2008})
$$-\delta\leq\frac{v_\gamma^{2+\epsilon}(y+a_*)-w_\gamma^{2+\epsilon}(y+a_*)}{\mathbb{P}(X_1>-y-a_*)w^{1+\epsilon}_\gamma(y+a_*)},\quad \forall y\leq s(x).$$
Using the definition of $w_\gamma$ and noticing the non-negativity of $W_\gamma$, \eqref{sec2eq6} implies particularly that $w(y)-v(y)=\omicron(w(y))$, as $y\to-\infty$. Therefore, such $a_*$ can be found.

\bibliography{bib}
\bibliographystyle{apacite}

\end{document}